\documentclass[12pt]{amsart}
\usepackage[utf8]{inputenc}
\usepackage{geometry}
\usepackage{amsfonts}
\usepackage{amsmath}
\usepackage{mathtools}
\usepackage{mathrsfs}
\usepackage{amsthm}
\usepackage{amssymb}
\usepackage{graphicx}
\usepackage{bbm}
\usepackage[none]{hyphenat}
\usepackage{multirow}
\usepackage{color}	
\usepackage{enumerate}
\usepackage{multicol}
\usepackage{subfigure}
\usepackage{float}
\usepackage{xcolor}
\usepackage{enumitem}
\usepackage{hyperref}
\usepackage{tikz-cd}

\DeclareMathOperator{\sym}{Sym}
\DeclareMathOperator{\diag}{diag}

\newtheorem{thm}{Theorem}[section]
\newtheorem{lem}[thm]{Lemma}
\newtheorem{prop}[thm]{Proposition}
\newtheorem{cor}[thm]{Corollary}
\theoremstyle{remark}

\numberwithin{equation}{section}

\title[Symmetric powers over function fields]{Cuspidality criterion for symmetric powers \\ of automorphic representations of GL(2) \\ over function fields}

\author[L. Lomel\'i]{Luis Lomel\'i}
\author[J. Navarro]{Javier Navarro}

\date{\today}


\makeatletter
\def\namedlabel#1#2{\begingroup
    #2%
    \def\@currentlabel{#2}%
    \phantomsection\label{#1}\endgroup
}
\makeatother

\let\olditemize\enumerate
\def\enumerate{\olditemize\itemsep=0pt}

\begin{document}

\begin{abstract}
Given a cuspidal automorphic representation of GL(2) over a global function field, we establish a comprehensive cuspidality criterion for symmetric powers. The proof is via passage to the Galois side, possible over function fields thanks to the Langlands correspondence of L. Lafforgue and additional results of G. Henniart and B. Lemaire. Our work is guided by the number fields results of Kim-Shahidi and Ramakrishnan.
\end{abstract}

\maketitle

\section*{Introduction}

Ever since the early beginnings of the Langlands Program, symmetric powers of a cuspidal representation $\pi$ of ${\rm GL}_2(\mathbb{A}_F)$ over a global field $F$ with ring of adèles $\mathbb{A}_F$ have earned a special place in the theory of automorphic forms and representations; Langlands himself noted that their automorphy implies the Ramanujan Conjecture for general linear groups, to say the least.
The literature on the subject is very general and spread over a diverse set of articles. We here have the particular interest of establishing a complete cuspidality criterion for ${\rm Sym}^n\pi$ over a global function field.

Gelbart and Jacquet established the adjoint lift of $\pi$ to ${\rm GL}_3(\mathbb{A}_F)$ in the generality of a global field in \cite{GeJa1978}, namely ${\rm Ad}\,\pi = {\rm Sym}^2\pi \otimes \omega_\pi^{-1}$, where $\omega_\pi$ is the central character of $\pi$. Then ${\rm Sym}^2\pi$ is automorphic and can be written as an isobaric sum. The cuspidality criterion for the symmetric square is already present in their work, in particular, ${\rm Sym}^2\pi$ is cuspidal if and only if $\pi$ is dihedral.

In the number field case, the automorphy of ${\rm Sym}^3\pi$ was established by Kim and Shahidi in \cite{KiSh2002}, where it is also first noted that it is cuspidal unless $\pi$ is dihedral or tetrahedral. They obtain further breakthroughs in \cite{Ki2003,KiSh2002cusp}; Kim proving the automorphy of  ${\rm Sym}^4\pi$, then in Kim and Shahidi one can find a thorough cuspidality criterion including the observation that ${\rm Sym}^4\pi$ is non-cuspidal if $\pi$ is dihedral, tetrahedral or octahedral.  Assuming all $\sym^n\pi$ modular, and still in the number field case, Ramakrishnan provides the criterion for ${\rm Sym}^5\pi$ and ${\rm Sym}^6\pi$ as well as noting that this is enough for higher symmetric powers \cite{Ra2009}. 

In contrast to the literature, we work with $\ell$-adic instead of complex representations on the Galois side, and tackle the function field case. In fact, from this point onwards, $F$ denotes a global function field of characteristic $p$ and $\ell$ is a prime different from $p$. In this scenario, the global Langlands correspondence is a landmark result of V. Drinfeld for ${\rm GL}_2$ \cite{Dr1978} and L. Lafforgue for ${\rm GL}_n$ \cite{LLaff2002}. In particular, as a consequence of their work, we know that ${\rm Sym}^n \pi$ is automorphic in positive characteristic. Thanks to the available machinery over function fields, our proofs are unconditional and we provide a comprehensive criterion including icosahedral representations.

Our main results on the automorphic side of the Langlands correspondence are summarized in Theorem \ref{thm:main:cuspidal}, which provides a cuspidality criterion for the symmetric powers. In order to phrase the criterion in a succint way, we let $M$ be the maximal power such that ${\rm Sym}^{M}\pi$ is cuspidal, writing $M = \infty$ in case every symmetric power of $\pi$ is cuspidal.

\medskip

\noindent{\bf Automorphic Criterion.} \emph{Let $\pi$ be a cuspidal representation of ${\rm GL}_2(\mathbb{A}_F)$. Then ${\rm Sym}^6\pi$ is cuspidal if and only if $M = \infty$. If ${\rm Sym}^6\pi$ is non-cuspidal, then $\pi$ admits the following classification.}
\begin{itemize}
    \item[] \emph{M=1: $\pi$ is dihedral.}
    \item[] \emph{M=2: $\pi$ is tetrahedral.}
    \item[] \emph{M=3: $\pi$ is octahedral.}
    \item[] \emph{M=5: $\pi$ is icosahedral.}
\end{itemize}
\emph{Additionally, ${\rm Sym}^4\pi$ is cuspidal if and only if ${\rm Sym}^5\pi$ is as well.}


\medskip

However, the proof of our automorphic criterion is via passage to the Galois side. Possible, thanks to the work of L. Lafforgue \cite{LLaff2002} as expanded to $\ell$-adic representations by Henniart and Lemaire in \cite{HeLe2011}, results that we summarize in \S~\ref{pass:to:auto}. Given a cuspidal $\pi$ of ${\rm GL}_2(\mathbb{A}_F)$ corresponding to an irreducible $2$-dimensional $\ell$-adic Galois $\sigma$ via the global Langlands correspondence
\[ \pi \longleftrightarrow \sigma, \]
the automorphic representation ${\rm Sym}^n\pi$ is cuspidal if and only if ${\rm Sym}^n\sigma$ is irreducible.
With this in mind, our main result on the Galois side is a reducibility criterion, where $M$ in this setting denotes the maximal irreducible symmetric power. It is  Theorem~\ref{thm:symm:Galois}, whose contents are as follows.

\medskip

\noindent{\bf Galois criterion.} \emph{Let $\sigma$ be an irreducible 2-dimensional $\ell$-adic Galois representation, then the following are equivalent:}
\begin{itemize}
    \item[(i)] \emph{$\sigma$ has finite image in ${\rm PGL}_2(\overline{\mathbb{Q}}_\ell)$.}
    \item[(ii)] \emph{There exists an integer $n\geq 2$ such that $\,{\rm Sym}^n\sigma$ is reducible.}
    \item[(iii)] \emph{${\rm Sym}^6\sigma$ is reducible.}
\end{itemize}
\emph{Specifically, we have the following classification when these properties are met.}
\begin{itemize}
    \item[] \emph{M=1: $\sigma$ is dihedral.} 
	\item[] \emph{M=2: $\sigma$ is tetrahedral.} 
	\item[] \emph{M=3: $\sigma$ is octahedral}. 
	\item[] \emph{M=5: $\sigma$ is icosahedral.} 
\end{itemize}
\emph{Additionally, ${\rm Sym}^5\sigma$ is reducible if and only if ${\rm Sym}^4\sigma$ is as well.}

\medskip

Let us make a few comments around the statement of the criterion. First, from general representation theoretical results, we deduce that the heart of the irreducibility criteria is contained in the cases of $n=2,3,4$ and $6$. We next prove that if ${\rm Sym}^n \sigma$ is reducible then there exists an open subgroup $H$ of $G_F$ such that the restriction of $\sigma$ to $H$ is reducible, say $\sigma_H=\mu_1\oplus \mu_2$ for $\ell$-adic characters $\mu_1$ and $\mu_2$. Now, from observations made by Henniart-Lemaire on $\ell$-adic representations \cite{HeLe2011}, $\mu_1$ and $\mu_2$ have open kernels in the Weil group of the extension of $F$ attach to $H$. We then prove that the image $J$ of $\sigma$ in ${\rm PGL}_2(\overline{\mathbb{Q}}_\ell)$, considering $\sigma$ as a homomorphism into ${\rm GL}_2(\overline{\mathbb{Q}}_\ell)$, is a finite group. In this case, $\sigma$ is defined to be dihedral, tetrahedral, octahedral, or icosahedral, according to $J$ being one of these finite groups, and observe that it cannot be abelian due to the irreducibility of $\sigma$. We note that this classification depends only on the isomorphism class of $\sigma$.

Let us now present the contents of the article in a more detailed fashion. We begin with a section on the preliminaries. In particular, we recall the representation theoretic Clebsch-Gordan formulas in \S~\ref{clebshgordan}. We proceed in \S~\ref{sec:Galois:crit:I} to prove two general useful lemmas concerning the reducibility of symmetric powers. First, if $\sym^N(\sigma)$ is reducible for some positive integer $N$, then it is reducible hence forwards for $n \geq N$. This tells us that the maximal irreducible symmetric power $M$, when it exists as a positive integer, is indeed determined by the fact that $\sym^M(\sigma)$ is irreducible, while $\sym^{M+1} (\sigma)$ is reducible. When $M$ is finite, the second lemma further reduces to $M<6$. However, there are representations for which $M=\infty$, i.e., every $\sym^n(\sigma)$ is irreducible, cf. Igusa \cite{igusa1959fibreIII}.



For $n<6$, the irredubility criteria for $\sym^n(\sigma)$ are examined in \S~\ref{sec:Galois:crit:II}. As we shall see by passing to the automorphic side in \S~\ref{pass:to:auto}, our cuspidality criteria align with the well established results of Gelbart-Jacquet over global fields and Kim-Shahidi over number fields. For general $n$, we are very much influenced by the number fields approach of Ramakrishnan.

Specifically, for the case of $n=2$, we show that $\sym^2 (\sigma)$ is reducible if and only if $\sigma\cong\sigma\otimes\chi$ for some non-trivial quadratic character $\chi$. In this case, we find that $H=\ker \chi$ leads to $\sigma_H$ being reducible. Additionally, the image of $H$ in ${\rm PGL}_2(\overline{\mathbb{Q}}_\ell)$ is an abelian index-two subgroup of $J$, implying that $J$ must be dihedral.

Moving on to the case of $n=3$, we establish that if $\sigma$ is non-dihedral, then $\sym^3 (\sigma)$ is reducible if and only if $\sym^2 (\sigma)\cong\sym^2(\sigma)\otimes\mu$ for some non-trivial cubic character $\mu$. In this situation, if $H'=\ker \mu$, then $\sigma_{H'}$ is dihedral, leading to the existence of an index-two subgroup $H$ of $H'$ such that $\sigma_H$ is reducible. We observe that the image of $H'$ in ${\rm PGL}_2(\overline{\mathbb{Q}}_\ell)$ is an index-three dihedral subgroup of $J$, then $J$ is tetrahedral.

Finally, for $n=4$, we prove that if $\sigma$ is neither dihedral nor tetrahedral, then $\sym^4 (\sigma)$ is reducible if and only if there exists a non-trivial quadratic character $\chi$ such that $\sym^3 (\sigma)\cong\sym^3(\sigma)\otimes\chi$. In this case, the kernel $H'=\ker \chi$ results in $\sigma_{H'}$ being tetrahedral. Consequently, there exists an index-six subgroup $H$ of $H'$ such that $\sigma_H$ is reducible. In this case, the group $J$ is necessarily octahedral. It is noteworthy that in each of these cases, the image of $H$ in ${\rm PGL}_2(\overline{\mathbb{Q}}_\ell)$ is an abelian group appearing in a subnormal series with abelian factors of $J$, which reflects the solvability of the dihedral, tetrahedral and octahedral groups. 

The non-solvable icosahedral case requires a separate treatment. We explore this case in \S~\ref{sec:Galois:crit:III} and take the opportunity to gather the remaining reducibility Galois criteria that involve $\sym^6(\sigma)$. There we assume $\sigma$ is not solvable polyhedral, i.e., neither dihedral, tetrahedral nor octahedral. We begin by proving that $\sym^6(\sigma)$ is reducible if and only if $\sym^5(\sigma)\cong {\rm Ad}(\sigma)\otimes\sigma'$, for another irreducible 2-dimensional $\ell$-adic representation $\sigma'$. The existence of the representation $\sigma'$ was expected in characteristic $p$ since it appears in the characteristic zero case in the literature, cf. \cite{Ki2004}, \cite{Ra2009} and \cite{Wang2003}. We then prove, and this is crucial, that $\sym^3(\sigma) \cong \sym^3(\widetilde\sigma)$, where $\widetilde\sigma = \sigma'\otimes\xi$, for some character $\xi$. Furthermore, for suitable bases, there exists an isomorphism $\sym^3(\sigma) \xrightarrow{\, \sim\, } \sym^3(\widetilde\sigma)$ of the form $\sym^3(g)$, for some $g\in{\rm GL}_2(\overline{\mathbb{Q}}_\ell)$. Then, using that the homomorphism $\sym^3\colon {\rm GL}_2(\overline{\mathbb{Q}}_\ell)\to{\rm GL}_4(\overline{\mathbb{Q}}_\ell)$ has finite kernel, we deduce that there exists an open subgroup $H'$ (which shall be denoted by $H$ in \S~\ref{sec:Galois:crit:III}) of $G_F$ such that $\sigma_{H'}\cong\widetilde\sigma_{H'}$. From the relation $\sym^5(\sigma)\cong {\rm Ad}(\sigma)\otimes\sigma'$, we conclude that $\sym^5(\sigma_{H'})$ is reducible, obtaining in this way an open subgroup $H$ of $G_F$ such that $\sigma_H$ is reducible.

We conclude with a passage from irreducible Galois $\sigma$ to cuspidal automorphic $\pi$ via the global Langlands correspondence of L. Lafforgue \cite{LLaff2002}, incorporating the results of Henniart-Lemaire \cite{HeLe2011}. We do this in \S~\ref{pass:to:auto}, after setting up the preliminaries. This enables us to obtain the Automorphic Criterion mentioned above, as well as more delicate cuspidality criteria for $\sym^n(\pi)$ corresponding in turn to the Galois treatment for the cases when $n<6$ and the separate discussion involving the case of $n=6$ and icosahedral representations.

\subsection*{Acknowledgments} We thank Guy Henniart for mathematical communications. The first author is grateful to the Institut des Hautes \'Etudes Scientifiques for the hospitality provided during a summer visit in 2023, when this article was finalized. He was supported in part by FONDECYT Grant 1212013. The second author was partially supported by USM PIIC Initiation to Scientific Research Program.

\section{Preliminaries}\label{prelim}

\subsection{}
We let $F$ be a global function field of characteristic $p$ and fix a separable closure $\overline{F}$. We denote by $\mathbb{A}_F$ the ring of adèles of $F$. We also fix a prime number $\ell\not = p$ and an algebraic closure $\overline{\mathbb{Q}}_\ell$ of the $\ell$-adic numbers $\mathbb{Q}_\ell$.  We write $G_F$ and $\mathcal{W}_F$ for the absolute Galois group ${\rm Gal}(\overline{F}/F)$ and the Weil group of $F$ corresponding to $\overline{F}$, respectively. We assume all separable field extensions $E/F$ lie inside $\overline{F}$. 

Let $\tau$ be either a homomorphism $G_F\to {\rm GL}(V)$ or $\mathcal{W}_F\to {\rm GL}(V)$, where $V$ is a $\overline{\mathbb{Q}}_\ell$-vector space endowed with the topology induced from $\overline{\mathbb{Q}}_\ell$. We say that $\tau$ is an $\ell$-adic representation of $G_F$, resp. of $\mathcal{W}_F$, if it is continuous, unramified at almost every place of $F$, and defined over a finite extension of $\mathbb{Q}_\ell$ (cf. \S~IV.2.1 of \cite{HeLe2011}). If $H$ is a subgroup of $G_F$, we use $\tau_H$ to denote the restriction $\tau|_H$. If $E/F$ is a separable field extension, we write $\tau_E$ for the restriction $\tau_{G_E}$.

We can view a finite-dimensional $\ell$-adic representation $\tau\colon G_F\to {\rm{GL}}(V)$ as a continuous homomorphism $\tau\colon G_F\to {\rm{GL}}_m(\overline{\mathbb{Q}}_\ell)$, where $m$ is the dimension of $V$, after choosing a basis. For every integer $n\geq 1$, we have the $n$-th symmetric power map
\[ \sym^n\colon{\rm{GL}}_m(\overline{\mathbb{Q}}_\ell)\to {\rm{GL}}_N(\overline{\mathbb{Q}}_\ell), \text{ where } N=\left( \begin{array}{c}
         m+n-1  \\
          n
    \end{array}\right). \]
It is well known, since Chevalley, that if $\tau$ is semisimple then $\sym^n(\tau)$ is semisimple.

Additionally, we introduce the following notation:
\[ A^n(\tau) = \sym^n(\tau) \otimes \omega_\tau^{-1}, \text{ where } \omega_\tau = \det(\tau). \]
In the particular case when $m=2$, we shall write ${\rm{Ad}}(\tau)$ instead of $A^2(\tau)$,
\begin{equation*}
   \begin{tikzcd}[column sep=small]
   {\rm GL}_2(\overline{\mathbb{Q}}_\ell) \arrow[rr,"\textcolor{black}{{\rm Ad}}"] & & {\rm GL}_3(\overline{\mathbb{Q}}_\ell) \\
	&	{\rm PGL}_2(\overline{\mathbb{Q}}_\ell) \arrow[ul,leftarrow] \arrow[ur]
   \end{tikzcd}
\end{equation*}
since it is the adjoint map from ${\rm GL}_2$ to ${\rm GL}_3$ of Gelbart-Jacquet \cite{GeJa1978}.
We further observe that
\begin{equation}\label{eq:A1:dual} 
A^1(\tau) = \tau\otimes \omega_\tau^{-1}\cong \tau^{\vee},
\end{equation}
where $\tau^\vee$ denotes the contragradient representation of $\tau$.

\subsection{}\label{clebshgordan}
Let us recall the Clebsch-Gordan formulas for symmetric powers of ${\rm{GL}}_2(\overline{\mathbb{Q}}_\ell)$. Start by taking $\sym^1 = {\rm{Id}}$. Next, if $g\in {\rm{GL}}_2(\overline{\mathbb{Q}}_\ell)$ and $i$, $m$ are non-negative integers such that $i\leq m/2$, then 
    \begin{equation*}
        \sym^i(g)\otimes\sym^{m-i}(g)\cong \bigoplus_{j=0}^i\sym^{m-2j}(g)\otimes \det(g)^j.
    \end{equation*}
Here the isomorphism is obtained via matrix conjugation.

\subsection{}\label{repfacts}
We gather here some useful representation theoretical facts, valid for any group $G$ and finite-dimensional representations over an algebraically closed field of characteristic zero. To be specific, let $(\sigma,V)$, $(\rho,W)$ and $(\tau, U)$ be finite-dimensional representations of $G$ over $\overline{\mathbb{Q}}_\ell$. Let us denote by ${\rm hom}_G(\rho,\tau)$ the representation of $G$ on the vector space ${\rm Hom}_{\overline{\mathbb{Q}}_\ell}(W,U)$ given by 
\[
(gT)(w)=\tau(g)(T(\sigma(g^{-1})(w)), \quad g\in G, T\in {\rm Hom}_{\overline{\mathbb{Q}}_\ell}(W,U), w\in W.
\]
The map 
\[
\begin{split}
	{\rm Hom}_G(\sigma\otimes \rho, \tau) & \longrightarrow {\rm Hom}_G(\sigma, {\rm hom}_G(\rho,\tau)) \\
	f & \longmapsto f^*\colon v \mapsto (w \mapsto f(v\otimes w))
\end{split}
\]
is a (natural) isomorphism. On the other hand, the bilinear map 
\[
\begin{split}
	W^\vee\times U & \longrightarrow {\rm Hom}_{\overline{\mathbb{Q}}_\ell}(W,U) \\
	(T,u) & \longmapsto T_u \colon w\mapsto T(w) u
\end{split}
\]
gives rise to a (natural) $G$-isomorphism 
\[
\rho^\vee\otimes \tau \xrightarrow{\ \ \sim\ \ } {\rm hom}_G(\rho,\tau).
\]
Thus, we have an isomorphism 
\begin{itemize}
\item[\namedlabel{eq:prel:Hom_iso}{(1.2)}]
${\rm Hom}_G(\sigma\otimes \rho, \tau) \xrightarrow{\ \ \sim \ \ } {\rm Hom}_G(\sigma, \rho^\vee\otimes \tau). $
\end{itemize}
And, as a consequence we have the following property: 

\begin{itemize}
\item[\namedlabel{item:app:3.2}{(1.3)}] Let $\rho,\sigma$ and $\tau$ be semisimple finite-dimensional representations of $G$, with $\sigma$ irreducible. If $\tau$ is a subrepresentation of $\rho \otimes \sigma$, then $\sigma$ is a subrepresentation of $\rho^\vee\otimes \tau$.
\end{itemize}

\subsection{}\label{clasification:repsGL2}

We work with a given $2$-dimensional irreducible $\ell$-adic Galois representation $\sigma$. When viewed as a homomorphism $\sigma \colon G_F \to {\rm GL}_2(\overline{\mathbb{Q}}_\ell)$, we let ${\rm proj}$ be the canonical projection from ${\rm GL}_2(\overline{\mathbb{Q}}_\ell)$ to ${\rm PGL}_2(\overline{\mathbb{Q}}_\ell)$.
\begin{equation*}
   \begin{tikzcd}[column sep=small]
   G_F \arrow[r,"\textcolor{black}{\sigma}"] & {\rm GL}_2(\overline{\mathbb{Q}}_\ell) \arrow[d,"\textcolor{black}{{\rm proj}}"] \\
	&	{\rm PGL}_2(\overline{\mathbb{Q}}_\ell) \arrow[ul,leftarrow,"\textcolor{black}{{\rm proj} \cdot \sigma}"]
   \end{tikzcd}
\end{equation*}

Let $J$ be the image ${\rm proj}(\sigma(G_F))$. In contrast to the complex finite-dimensional case, where a continuous homomorphism $G_F\to{\rm GL}_m(\mathbb{C})$ always has open kernel, this may not be the case for $\ell$-adic representations. In particular, $J$ may or may not be finite.

In the case of finite image, we define $\sigma$ to be dihedral, tetrahedral, octahedral, or icosahedral, according to $J$ being one of these finite groups. We note that this classification depends only on the isomorphism class of $\sigma$. One of our main results shows that $J$ being finite is equivalent to $\sigma$ having open kernel, and in turn equivalent to $\sym^n(\sigma)$ being reducible for some $n$. 

We adopt a similar terminology for $\ell$-adic representations of the Weil group. Namely, if $\sigma\colon \mathcal{W}_F\to{\rm GL}_2(\overline{\mathbb{Q}}_\ell)$ is an irreducible $\ell$-adic representation such that ${\rm proj}(\sigma(\mathcal{W}_F))$ is finite, then we define $\sigma$ to be dihedral, tetrahedral, octahedral, or icosahedral, depending on ${\rm proj}(\sigma(\mathcal{W}_F))$. These definitions are in accordance to the complex case found in \cite{langlands96base}. We point out that if $\sigma \colon \mathcal{W}_F \to {\rm GL}_2(\mathbb{C})$ is a continuous irreducible representation then, as noted by Langlands in [\emph{loc.\,cit.}], the image ${\rm proj}(\sigma(\mathcal{W}_F))$ is a finite subgroup of ${\rm PGL}_2(\mathbb{C})$.

\section{Reducibility Criteria for Symmetric Powers of Galois Representations: General Notions} \label{sec:Galois:crit:I}

In this section we address two general lemmas on the irreducibility of the symmetric powers of an irreducible two-dimensional $\ell$-adic representation
\[ \sigma \colon G_F \to {\rm{GL}}(V). \]
Their proofs involve the general representation theoretic Clebsch-Gordan formulas recalled in \S~\ref{clebshgordan}, coupled with the reducibility results summarized in \S~\ref{repfacts}.

\subsection{}
The first of the two lemmas allows us to make precise the notion of maximal irreducible symmetric power of $\sigma$. We write $M = \infty$ in case every symmetric power of $\sigma$ is irreducible. Otherwise, it is the positive integer $M$ for which ${\rm Sym}^{M}(\sigma)$ is irreducible, while ${\rm Sym}^{M+1}(\sigma)$ is reducible.

\begin{lem}\label{lem:redsymm:geqn}
Assume that $\sym^N(\sigma)$ is reducible for a given integer $N>1$. Then $\sym^n(\sigma)$ is reducible for all $n\geq N$. 
\end{lem}

\begin{proof}
    We assume $\sym^n(\sigma)$ is reducible, and we know that it is semisimple, hence we can write $\sym^n(\sigma)=\sigma_1\oplus \sigma_2$, for subrepresentations $\sigma_1$ and $\sigma_2$. By Clebsch-Gordan, we have that
    \begin{equation*}
        \sigma\otimes \sym^{n+1}(\sigma)=\sym^{n+2}(\sigma)\oplus \sym^n(\sigma)\otimes \omega_\sigma.
    \end{equation*}
    Then  $\sigma_i\otimes \omega_\sigma$ is a subrepresentation of $\sigma\otimes \sym^{n+1}(\sigma)$ for $i=1,2$. 
    
    If $\sym^{n+1}(\sigma)$ were irreducible, then $\sym^{n+1}(\sigma^\vee)$ would be a subrepresentation of $\sigma\otimes \sigma_i^\vee\otimes \omega_\sigma^{-1}\cong (\sigma\otimes\sigma_i)^\vee$ for each $i=1,2$, by Property \ref{item:app:3.2}. In particular, 
    \[ \dim \sigma\otimes\sigma_i=2\dim \sigma_i \geq \dim \sym^{n+1}(\sigma)=n+2, \] for $i=1,2$, but 
    \[ \dim \sigma_1 +\dim \sigma_2 =\dim \sym^n(\sigma)=n+1. \] 
Therefore, $\sym^{n+1}(\sigma)$ must also be reducible.
\end{proof}

An interesting fact is the existence of representations with $M=\infty$ and whose corresponding image $J$ in ${\rm PGL}_2(\overline{\mathbb{Q}}_\ell)$ is infinite. For instance, for $\ell$-adic representations associated with certain elliptic curves defined over $F$, Igusa proved in \cite{igusa1959fibreIII} an analogous statement to the well known Serre's open image theorem \cite{Se1972}.

\subsection{}
The second lemma significantly reduces the number of cases required in the reducibility criterion.

\begin{lem}\label{lem:red:symm}
    If $M$ is finite, then $M=1,2,3$ or $5$.
\end{lem}

\begin{proof}
If $\Pi$ is an irreducible subrepresentation of $\sym^{M+1}(\sigma)$ then, by Clebsch-Gordan, it is a subrepresentation of $\sigma\otimes\sym^M(\sigma)$. Since $\sym^M(\sigma)$ is irreducible, then $\sym^M(\sigma^\vee)$ is a subrepresentation of $\sigma\otimes\Pi^\vee$, by \ref{item:app:3.2}. In particular, we have that
\[ \dim(\sigma\otimes\Pi^{\vee})=2\dim \Pi\geq M+1. \]
This last inequality, together with $\dim \sym^{M+1}(\sigma)=M+2$, forces a decomposition into irreducible representations
\begin{equation*}
    \sym^{M+1}(\sigma)=\Pi_1\oplus \Pi_2 \text{ for } M>1,
\end{equation*}
where 
\[ \dim \Pi_2-\dim\Pi_1 =  
\left\{ \begin{array}{rl} 	
   1	& \text{if } M \text{ is odd,} \\
   0	& \text{if } M \text{ is even.}
\end{array} \right.  \]
    
If $M$ is odd, then $\dim(\sigma\otimes\Pi_2^{\vee})=M+3$. By the semisimplicity of $\sigma\otimes \Pi_2^\vee$, we have that
\[ \sigma\otimes\Pi_2^\vee = \sym^M(\sigma^\vee)\oplus\tau \]
for some 2-dimensional representation $\tau$. Now, since $\Pi_2$ is irreducible and $\tau$ is a subrepresentation of $\sigma\otimes\Pi_2^\vee$, then $\Pi_2$ is a subrepresentation of $\sigma\otimes\tau^\vee$, again by \ref{item:app:3.2}. This implies in particular that
\[ \dim  \Pi_2=\frac{M+3}{2}\leq 4 \ \Longrightarrow \ M\leq 5. \] 

If  $M$  is even, then $\dim (\sigma\otimes \Pi_1^{\vee})=\dim (\sigma\otimes\Pi_2^{\vee})=M+2$. In this case, there exists a character $\mu$ such that 
\[ \sigma\otimes\Pi_1^\vee = \sym^M(\sigma^\vee)\oplus\mu. \]
Applying once again \ref{item:app:3.2}, we deduce that $\sigma^\vee$ is a subrepresentation of $\Pi_1^\vee\otimes\mu^{-1}=(\Pi_1\otimes\mu)^\vee$. Since $\Pi_1\otimes\mu$ is irreducible, then $\sigma\cong \Pi_1\otimes\mu$; in particular, $\dim \Pi_1=2$ and then we necessarily have $M=2$ in this case. 
\end{proof}

Notice that in the proof of Lemma \ref{lem:red:symm} we also allow the case of $M=1$, which is an actual possibility as we will see in Theorem \ref{thm:sym2}. In fact, at the end of our study we will be able to conclude the same for $n=2$, $3$ and $5$.

\section{Reducibility Criteria for Symmetric Powers of \\ Galois Representations: ${\sym^n}(\sigma)$, $n < 6$} \label{sec:Galois:crit:II}

The cases treated in this section are well established over number fields. We are particularly influenced by \cite{GeJa1978, Ra2000, Ra2009} for the symmetric square, \cite{KiSh2002, KiSh2002cusp, Ki2003} for the symmetric cube and fourth, and \cite{Ra2009} for the symmetric fifth. In contrast to the literature, we work with $\ell$-adic instead of complex representations, and tackle the case of a function field $F$ of characteristic $p$, with $\ell \neq p$. Specifically, all of the representations that we will encounter in this section and the next are over the field $\overline{\mathbb{Q}}_\ell$.

\subsection{Symmetric Square}\label{sec:sym2}

In this case, the following theorem and its corollary present the reducibility criterion.

\begin{thm}\label{thm:sym2}
The following are equivalent
\begin{enumerate}
    \item[\namedlabel{item:square:i}{(i)}] $\sym^2(\sigma)$ is reducible. 
    \item[\namedlabel{item:square:ii}{(ii)}] $\sigma$ is dihedral. 
    \item[\namedlabel{item:square:iii}{(iii)}] $\sigma\cong \sigma\otimes\chi$ for some non-trivial character $\chi$. 
\end{enumerate}
\end{thm}

\begin{proof}
Let us suppose first that $\sigma$ is dihedral, meaning that ${\rm proj}(\sigma(G_F))$ is finite and dihedral. Since the dihedral groups have an index two subgroup which is cyclic, we have that $\sigma(H)$ is cyclic modulo the center of $\sigma(G_F)$, for some index two subgroup $H$ of $G_F$. In particular, $\sigma(H)$ is abelian. By Schur's lemma, $\sigma_H$ is reducible. Note that $\sigma$ being irreducible, the representation $\sigma_H$ is semisimple. Actually, if $c\in G_F\setminus H$,  then $\sigma_H\cong \mu\oplus \mu^c$, for some character $\mu$ of $H$, where 
\begin{equation*}
	\mu^c(h):= \mu(chc^{-1}), \quad h\in H.
\end{equation*}
Let $E/F$ be the quadratic extension associated to $H$, so that $H = G_E$.

Now, by Frobenius reciprocity 
\begin{equation*}
	{\rm Hom}_{G_E}(\sigma_E,\mu) = {\rm Hom}_{G_F}(\sigma, {\rm Ind}_E^F\mu).
\end{equation*}
Since $\sigma$ is irreducible, the above being not zero implies that $\sigma\cong {\rm Ind}_E^F\mu$. 
If $\chi$ is the only non-trivial character of $G_F$ which is trivial on $G_E$, then 
\begin{equation*}
    \sigma\otimes \chi \cong {\rm{Ind}}_E^F\mu \otimes \chi\cong {\rm{Ind}}_E^F(\mu\otimes \chi_E)={\rm{Ind}}_E^F(\mu)\cong \sigma,
\end{equation*}
and \ref{item:square:iii} is satisfied. Now suppose that \ref{item:square:iii} holds. By comparing determinants, we have that $\chi^2=1$. Then the kernel of $\chi$ is $G_E$, for some quadratic extension $E/F$. If $\varphi\colon \sigma \to \sigma\otimes \chi$  is an isomorphism,  then $\varphi$ is not a scalar, since $\chi$ is not trivial. Thus, by taking restriction to $G_E$ we get a non-scalar isomorphism
\begin{equation*}
    \varphi_E\colon \sigma_E\longrightarrow \sigma_E.
\end{equation*}
By Schur's lemma, we have that $\sigma_E$ is reducible. As above, we have that $\sigma_E\cong \mu \oplus \mu^c$ for some character $\mu$ of $G_E$, where $c$ is the non-trivial element of $G_F/G_E\cong{\rm{Gal}}(E/F)$. In order to prove that $\sigma$ is dihedral, it is enough to prove that ${\rm proj}(\sigma(G_F))$ is finite, since the dihedral groups are the only finite non-abelian subgroups of ${\rm PGL}_2(\overline{\mathbb{Q}}_\ell)$ with index two abelian subgroups. For this, it suffices to verify that 
\begin{equation}\label{eq:img:mu:W_E}
\left\lbrace \frac{\mu(g)}{\mu^c(g)} : g\in \mathcal{W}_E\right\rbrace    
\end{equation}
is finite, where $\mathcal{W}_E$ is the Weil group of $E$, since $G_E$ is of finite index in $G$ and $\mathcal{W}_E$ is dense in $G_E$. We observe that $\sigma$ being unramified almost everywhere, the character $\mu$ is unramified at almost every place of $E$. Thus, $\mu$ is an $\ell$-adic character of $G_E$ and, in particular, the restriction $\mu\vert_{\mathcal{W}_E}$ is an $\ell$-adic character of $\mathcal{W}_E$. Now, via Class Field Theory, \eqref{eq:img:mu:W_E} is equal to the image of $\mu$ restricted to 
\begin{equation}
    \left\lbrace \frac{g}{c(g)} : g\in E^\times \backslash\mathbb{A}_E^\times \right\rbrace = \left\lbrace g\in E^\times \backslash\mathbb{A}_E^\times : N_{E/F}(g)=1 \right\rbrace,
\end{equation}
where $N_{E/F}\colon E^\times \backslash\mathbb{A}_E^\times\to F^\times \backslash\mathbb{A}_F^\times$ is the norm map. The latter group being compact, the finitness of its image under $\mu$ follows from the fact that $\mu\vert_{\mathcal{W}_E}$ has open kernel (see Lemme IV.2.7 of \cite{HeLe2011}).

Now let us prove that (the equivalent) \ref{item:square:ii} and \ref{item:square:iii} imply \ref{item:square:i}. Let $E/F$ be a quadratic extension and $\mu$ a character of $G_E$ such that $\sigma={\rm{Ind}}_E^F\mu$. In this case, for a suitable basis of $V$, $\sigma$ is as follows  
\begin{equation}\label{eq:sigma:dihedral}
    \sigma\colon g\longmapsto \left\lbrace \begin{array}{ll}
      \left( \begin{array}{cc}
         \mu(g)   & 0 \\
        0    & \mu(cgc^{-1})
       \end{array}\right)  & \mbox{if } g\in G_E, \\
        & \\
        \left( \begin{array}{cc}
         0 & \mu(cg) \\
         \mu(gc^{-1}) & 0
       \end{array}\right)  & \mbox{if } g\not\in G_E.
    \end{array} \right.
\end{equation}

From which we have that 
\begin{equation*}
    \sym^2(\sigma)\colon g\longmapsto \left\lbrace \begin{array}{ll}
      \left( \begin{array}{ccc}
         \mu^2(g)   & 0 & 0\\
        0    & \mu(g)\mu(cgc^{-1}) & 0 \\
        0 & 0 & \mu^2(cgc^{-1})
       \end{array}\right)  & \mbox{if } g\in G_E, \\
        & \\
        \left( \begin{array}{ccc}
         0 & 0 & \mu^2(cg) \\
         0 & \mu(cg)\mu(gc^{-1}) & 0\\ 
         \mu^2(gc^{-1}) & 0 & 0
       \end{array}\right)  & \mbox{if } g\not\in G_E.
    \end{array} \right.
\end{equation*}
From the above we readily see that $\sym^2(\sigma)$ is reducible. Actually, it has the one-dimensional subrepresentation given by 
\begin{equation*}
    \eta\colon g\longmapsto \left\lbrace \begin{array}{ll}
      \mu(g)\mu(cgc^{-1})=\det(\sigma(g)) & \mbox{if } g\in G_E, \\
        & \\
        \mu(cg)\mu(gc^{-1})=-\det (\sigma(g))  & \mbox{if } g\not\in G_E.
    \end{array} \right.
\end{equation*}
We observe that $\eta= \omega_\sigma \chi$, where $\omega_\sigma=\det\circ\ \sigma$ and $\chi$ is the quadratic character associated to $E$. Here we obtain
\begin{equation}\label{eq:sym2:dec}
    \sym^2(\sigma)\cong \omega_\sigma \chi \oplus {\rm{Ind}}_E^F \mu^2.
\end{equation}

Finally, let us suppose that \ref{item:square:i} holds. Since $\sym^2(\sigma)$ is semisimple and 3 -dimensional, it must have a one-dimensional summand, say $\eta$. By Clebsch-Gordan 
\begin{equation}\label{eq:sigma_x_sigma}
    \sigma\otimes \sigma = \sym^2(\sigma) \oplus \omega_\sigma, 
\end{equation}
we have that $\eta$ is a subrepresentation of $\sigma\otimes\sigma$. By \ref{item:app:3.2}, this implies that $\sigma^\vee$ is a subrepresentation of $\sigma\otimes\eta^{-1}$. Then, by irreducibility, we get
\begin{equation*}
    \sigma\otimes \eta^{-1}\cong \sigma^{\vee}\cong \sigma\otimes \omega_\sigma^{-1}, 
\end{equation*}
from which we infer that $\sigma\cong \sigma\otimes \eta \omega_\sigma^{-1}$. We observe that $\eta\not= \omega_\sigma$, since by \eqref{eq:prel:Hom_iso} we have 
\[
{\rm Hom}_G(\omega_\sigma, \sigma\otimes\sigma)\cong {\rm Hom}_G(\sigma^\vee\otimes\omega_\sigma, \sigma)\cong {\rm Hom}_G(\sigma,\sigma),
\] 
and the latter is one-dimensional. Thus, by letting $\chi=\eta \omega_\sigma^{-1}$, we have that \ref{item:square:iii} holds. 

\end{proof}

The following result is a general observation, beginning at $n=2$ because of Theorem \ref{thm:sym2} and is valid thereafter due to Lemma \ref{lem:redsymm:geqn}. In terms of the maximal irreducible symmetric power, this happens exactly when $M=1$.

\begin{cor}\label{cor:sym:dihedral:reducible}
The representation $\sigma$ is dihedral if and only if $\sym^n(\sigma)$ is reducible for $n \geq 2$.
\end{cor}

\subsection{Symmetric Cube}\label{sec:cube}

Given that $\sym^3(\sigma)$ is reducible when $\sigma$ is dihedral by Corollary \ref{cor:sym:dihedral:reducible}, we complete the reducibility criterion in this case by restricting ourselves to non-dihedral representations in the following theorem.

\begin{thm}\label{thm:A3}
Suppose that $\sigma$ is non-dihedral. Then $\sym^3(\sigma)$ is reducible if and only if there exists some non-trivial character $\mu$ such that
\[ \sym^2(\sigma)\cong \sym^2(\sigma)\otimes \mu; \]
the condition being equivalent to $\sigma$ being tetrahedral. When this is the case we have the decomposition
\begin{equation*}\label{eq:dec:A3}
    A^3(\sigma)\cong \sigma\otimes\mu \oplus \sigma\otimes\mu^2.
\end{equation*}
\end{thm}

\begin{proof}
Let us suppose first that $\sym^2(\sigma)\cong \sym^2(\sigma)\otimes \mu$ for some non-trivial character $\mu$. By Clebsch-Gordan 
\begin{equation}\label{eq:dec:sym2}
    \sigma\otimes \sym^2(\sigma)\cong \sym^3(\sigma) \oplus \sigma\otimes \omega_\sigma.
\end{equation}
Now, $\sigma\otimes \sym^2(\sigma)\cong\sigma\otimes \sym^2(\sigma)\otimes \mu$ by hypothesis, so the right-hand side of \eqref{eq:dec:sym2} is equivalent to its twist by $\mu$, then
\begin{equation*}
    \sym^3(\sigma) \oplus \sigma\otimes \omega_\sigma\cong (\sym^3(\sigma) \oplus \sigma\otimes \omega_\sigma)\otimes \mu\cong \sym^3(\sigma)\otimes \mu \oplus \sigma\otimes \omega_\sigma\mu.
\end{equation*}
Since $\sigma$ is non-dihedral and $\mu$ is non-trivial, by Theorem \ref{thm:sym2} we have that $\sigma\otimes \omega_\sigma\not\cong\sigma\otimes \omega_\sigma\mu$. Thus, $\sigma\otimes \omega_\sigma\mu$ must be a factor of $\sym^3(\sigma)$. Hence, $\sym^3(\sigma)$ is reducible. 

Now let us assume that $\sym^3(\sigma)$ is reducible. Let $\tau$ be a subrepresentation of $\sym^3(\sigma)$. By \eqref{eq:dec:sym2}, $\tau$ is a subrepresentation of $\sigma\otimes\sym^2(\sigma)$. Since $\sym^2(\sigma)$ is irreducible, then it is a subrepresentation of $\sigma^\vee\otimes\tau$. In particular, $\tau$ is not one-dimensional. Thus, we have that 
\begin{equation*}
    \sym^3(\sigma)=\tau_1\oplus \tau_2,
\end{equation*}
for some irreducbile two-dimensional representations $\tau_1$ and $\tau_2$. Since $\sigma^\vee\otimes\tau_1$ is semisimple 4-dimensional, and $\sym^2(\sigma)$ is irreducible 3-dimensional, then there exists a character $\eta_1$ such that
\begin{equation}\label{eq:sigma:tau_1}
\sigma^\vee\otimes\tau_1=\sym^2(\sigma)\oplus\eta_1.
\end{equation}
This implies that $\sigma$ is a subrepresentation of $\tau_1\otimes\eta_1^{-1}$ and then $\sigma\cong \tau_1\otimes \eta_1^{-1}$ by irreducibility, i.e., $\tau_1\cong \sigma\otimes \eta_1$.  

Similarly, we have that $\tau_2\cong\sigma\otimes\eta_2$ for some character $\eta_2$. We obtain then 
\begin{equation}\label{eq:dec2:sym3}
    \sym^3(\sigma)\cong\sigma\otimes\eta_1 \oplus \sigma\otimes \eta_2.
\end{equation}
Now, by the identity $\tau_1\cong\sigma\otimes\eta_1$ we get 
\begin{equation*}
\sigma^\vee\otimes\tau_1\cong\sigma^\vee\otimes\sigma\otimes\eta_1\cong \sym^2(\sigma)\otimes\eta_1\omega_\sigma^{-1}\oplus \eta_1.
\end{equation*}
Thus, by \eqref{eq:sigma:tau_1} we obtain 
\[
\sym^2(\sigma)\cong \sym^2(\sigma)\otimes \eta_1\omega_\sigma^{-1}.
\]
We observe that neither $\eta_1$ nor $\eta_2$ can be equal to $\omega_\sigma$, since, by using repeatedly \eqref{eq:prel:Hom_iso}, we have
\[
{\rm Hom}_G(\sigma\otimes\omega_\sigma, \sigma\otimes\sym^2(\sigma))\cong {\rm Hom}_G(\sym^2(\sigma^\vee),\sigma^\vee\otimes\sigma^\vee),
\]
and the latter must be one-dimensional by the irreducibility of $\sym^2(\sigma)$ and dimension reasons. Also, similarly, ${\rm Hom}_G(\sigma\otimes\eta_i, \sigma\otimes \sym^2(\sigma))$ is one-dimensional, from which we deduce $\eta_1\not= \eta_2$. Thus, if we let $\mu_i=\eta_i \omega_\sigma^{-1}$ for $i=1,2$, then $\mu_1\not=\mu_2$ are non-trivial characters such that
\begin{equation}\label{eq:sym2:xmu_1}
	\sym^2(\sigma)\cong \sym^2(\sigma)\otimes \mu_1\cong\sym^2(\sigma)\otimes \mu_2
\end{equation}
and
\begin{equation}\label{eq:dec2:A3}
	A^3(\sigma)\cong \sigma\otimes \mu_1 \oplus \sigma\otimes \mu_2.
\end{equation}
The first part of the theorem is proved, since $\mu_1$ and $\mu_2$ are not trivial.

We observe that \eqref{eq:sym2:xmu_1} implies in particular that $\mu_1$ has order three. If $E/F$ is the cubic extension which corresponds to $\mu_1$, then similarly as in \S\,\ref{sec:sym2} above, \eqref{eq:sym2:xmu_1} implies that $\sym^2(\sigma_E)$ is reducible. Note that since $E/F$ is cyclic of order three and $\sigma$ is two-dimensional irreducible, then $\sigma_E$ is irreducible. By Theorem \ref{thm:sym2}, $\sigma_E$ is dihedral. In particular, ${\rm proj}(\sigma(G_F))$ is finite. It must be tetrahedral, since $A_4$ is the unique non-abelian finite subgroup of ${\rm PGL}_2(\overline{\mathbb{Q}}_\ell)$ with a dihedral index three subgroup.  

Now, by taking contragradient of both sides of \eqref{eq:dec2:A3}, we have that 
\begin{equation*}
    A^3(\sigma)\otimes \omega_\sigma^{-1}\cong \sigma\otimes \omega_\sigma^{-1}\mu_1^{-1}\oplus \sigma\otimes \omega_\sigma^{-1}\mu_2^{-1},
\end{equation*}
i.e., 
\begin{equation*}
    A^3(\sigma)\cong \sigma\otimes\mu_1\oplus \sigma\otimes\mu_2\cong \sigma\otimes \mu_1^{-1}\oplus \sigma\otimes \mu_2^{-1}.
\end{equation*}
If $\sigma\otimes\mu_1\cong \sigma\otimes \mu_1^{-1}$, then $\sigma\otimes\mu_1^2\cong\sigma$, which is not possible since $\mu_1^2$ is non-trivial and $\sigma$ is non-dihedral. Then, necessarily $\sigma\otimes\mu_1\cong \sigma\otimes\mu_2^{-1}$. The latter implies that $\mu_1=\mu_2^{-1}$, since $\sigma$ is non-dihedral. If we take, for instance, $\mu=\mu_2$, then $A^2(\sigma)\cong A^2(\sigma)\otimes \mu$ and, since $\mu_1=\mu_2^{-1}=\mu_2^2$, \eqref{eq:dec2:A3} becomes 
\begin{equation*}
    A^3(\sigma)\cong \sigma\otimes\mu \oplus \sigma\otimes\mu^2.
\end{equation*}
Theorem \ref{thm:A3} is now completed. 
\end{proof}

Similar to Corollary \ref{cor:sym:dihedral:reducible}, the following result is a consequence of Theorem \ref{thm:A3} and Lemma \ref{lem:redsymm:geqn}, together with the definition of $M$.

\begin{cor}\label{cor:sym:tetrahedral:reducible}
The representation $\sigma$ is tetrahedral if and only if $M=2$.
\end{cor}

\subsection{Fourth and Fifth Symmetric Powers}\label{sec:fourth}

If $\sym^3(\sigma)$ is reducible then, by Lemma \ref{lem:redsymm:geqn}, so is $\sym^4(\sigma)$. To complete the criterion it suffices to inspect the case when the symmetric cube is irreducible in the next theorem.  

\begin{thm}\label{thm:sym4}
Assume that $\sym^3(\sigma)$ is irreducible. Then $\sym^4(\sigma)$ is reducible if and only if there exists a non-trivial quadratic character $\chi$ such that
\[ \sym^3(\sigma)\cong \sym^3(\sigma)\otimes \chi; \]
the condition being equivalent to $\sigma$ being octahedral. When this is the case we have the decomposition 
\begin{equation*}
    A^4(\sigma)\cong {\rm{Ind}}_E^F(\omega_{\sigma_E}\mu) \oplus \sym^2(\sigma)\otimes\chi,
\end{equation*}
where $E/F$ is the quadratic extension corresponding to $\chi$ by class field theory and $\mu$ is some cubic character of $G_E$.
\end{thm}

\begin{proof}
First, if $\sym^3(\sigma)\cong\sym^3(\sigma)\otimes\chi$, then clearly $A^3(\sigma)\cong A^3(\sigma)\otimes\chi$. By Clebsch-Gordan 
\begin{equation}\label{eq:sigma_x_A3}
    \sigma\otimes A^3(\sigma)\cong A^4(\sigma) \oplus \sym^2(\sigma).
\end{equation}
Then we have 
\begin{equation*}
    A^4(\sigma)\oplus \sym^2(\sigma)\cong A^4(\sigma)\otimes\chi \oplus \sym^2(\sigma)\otimes\chi.
\end{equation*}
Since $\sym^2(\sigma)\otimes\chi\not\cong \sym^2(\sigma)$, we have that $\sym^2(\sigma)\otimes\chi$ is a factor of $A^4(\sigma)$. In particular, $\sym^4(\sigma)$ is reducible.

Let us suppose now that $\sym^4(\sigma)$, and therefore $A^4(\sigma)$, is reducible. Since $A^3(\sigma)$ is assumed to be irreducible, then for every subrepresentation $\tau$ of $\sigma\otimes A^3(\sigma)$, $A^3(\sigma)$ is a subrepresentation of $\sigma^\vee\otimes\tau$. In particular, $\sigma\otimes A^3(\sigma)$ cannot have one-dimensional subrepresentations. Thus, from \eqref{eq:sigma_x_A3} we get that 
\begin{equation}\label{eq:A4:tau+Pi}
    A^4(\sigma)\cong \tau \oplus \Pi,
\end{equation}
for some irreducible representations $\tau$ and $\Pi$ of dimensions 2 and 3, respectively. By comparing dimensions, we deduce that there exists a two-dimensional (necessarily irreducible) representation $\tau_1$ such that 
\begin{equation*}
   \sigma^\vee \otimes\Pi \cong A^3(\sigma)\oplus \tau_1.
\end{equation*}
By \ref{item:app:3.2}, we have that $\Pi$ is a subrepresentation of $\sigma\otimes\tau_1$. By comparing dimensions once again, we see that there exists a character $\eta$ such that
\begin{equation}\label{eq:sigmaxtau1}
\sigma\otimes\tau_1\cong \Pi \oplus \eta.
\end{equation}
Therefore, we must have that $\tau_1\otimes\eta^{-1}\cong \sigma^\vee$, i.e., $\tau_1\cong \sigma\otimes\chi$, where $\chi=\eta \omega_\sigma^{-1}$.
 
 Now, we have
\begin{equation*}
    \sigma\otimes\tau_1\cong \sigma\otimes(\sigma\otimes\chi)\cong\sym^2(\sigma)\otimes\chi \oplus \omega_\sigma\chi.
\end{equation*}
Thus, by \eqref{eq:sigmaxtau1} we conclude that 
\begin{equation}\label{eq:Pi:Sym2_chi}
    \Pi\cong \sym^2(\sigma)\otimes\chi.
\end{equation}
Then, by \eqref{eq:sigma_x_A3}, \eqref{eq:A4:tau+Pi} and \eqref{eq:Pi:Sym2_chi} above, we obtain 
\begin{equation}\label{eq:sigma_x_A3:tau}
    \sigma\otimes A^3(\sigma)\cong \tau \oplus \sym^2(\sigma)\otimes\chi \oplus \sym^2(\sigma).
\end{equation}
We see from \eqref{eq:sigma_x_A3:tau} that $\chi$ is not trivial, since we are assuming that $A^3(\sigma)$ is irreducible and thus 
\[
{\rm Hom}_G(\sym^2(\sigma),\sigma\otimes A^3(\sigma))\cong {\rm Hom}_G(A^3(\sigma)^\vee, \sigma\otimes\sym^2(\sigma^\vee))
\]
is one-dimensional. By taking contragradient of both sides of \eqref{eq:sigma_x_A3:tau}, we have that 
\begin{equation*}
    \sigma\otimes A^3(\sigma)\otimes \omega_\sigma^{-2}\cong \tau\otimes \omega_{\tau}^{-1}\oplus \sym^{2}(\sigma)\otimes \omega_\sigma^{-2}\chi^{-1}\oplus \sym^2(\sigma)\otimes \omega_\sigma^{-2}.
\end{equation*}
Hence, we obtain 
\begin{equation*}
    \tau \oplus \sym^2(\sigma)\otimes\chi \oplus \sym^2(\sigma) \cong \tau\otimes \omega_{\tau}^{-1}\omega_\sigma^2\oplus \sym^{2}(\sigma)\otimes \chi^{-1}\oplus \sym^2(\sigma).
\end{equation*}
In particular, we have that $\sym^2(\sigma)\otimes\chi\cong \sym^2(\sigma)\otimes \chi^{-1}$. This implies that $\chi$ is a non-trivial quadratic character by Theorem \ref{thm:A3}, since we are assuming that $A^3(\sigma)$ is irreducible. 

Let $E/F$ be the quadratic extension obtained from $\chi$ via class field theory. By restricting to $G_E$, we have from \eqref{eq:sigma_x_A3:tau} that
\begin{equation}\label{eq:sigma_x_A3:E}
    \sigma_E\otimes A^3(\sigma_E)\cong \tau_E\oplus \sym^2(\sigma_E)\oplus \sym(\sigma_E).
\end{equation}
This implies that $A^3(\sigma_E)$ is reducible. In fact, as we have observed, otherwise ${\rm Hom}_G(\sym^2(\sigma_E),\sigma_E\otimes A^3(\sigma_E))$ would be one-dimensional.
\noindent

The representation $\sigma$ being non-dihedral, $\sigma_E$ is irreducible. Thus, by \S\S \ref{sec:sym2}-\ref{sec:cube}, we have that ${\rm proj}(\sigma(G_E))$ is finite. Since $E/F$ is finite, ${\rm proj}(\sigma(G_F))$ is finite. Now, the representation $\sigma_E$ cannot be dihedral, since none of the groups $A_4$, $S_4$ and $A_5$ has an index-two subgroup which is dihedral. Actually, we necessarily have that $\sigma$ is octahedral and $\sigma_E$ is tetrahedral. We then can apply Theorem \ref{thm:A3}. In particular, we have that 
\begin{equation}\label{eq:dec:A3:E}
    A^3(\sigma_E)\cong \sigma_E\otimes\mu \oplus \sigma_E\otimes\mu^2,
\end{equation}
for some character $\mu$ of $G_E$. Let $s$ be the non-trivial element of $G_F/G_E\cong {\rm{Gal}}(E/F)$. Clearly, $s$ acts trivially on $A^3(\sigma_E)$, and cannot act trivially on its factors $\sigma_E\otimes\mu$ and $\sigma_E\otimes\mu^2$, since in this case $A^3(\sigma)$ would be reducible. Then, we have that $(\sigma_E\otimes\mu)^s=\sigma_E\otimes\mu^2$, i.e., $\mu^s=\mu^2$. This implies that $A^3(\sigma)\cong {{\rm{Ind}}}_E^F(\sigma_E\otimes\mu)$. Since $E/F$ is the extension defined by $\chi$, we obtain
\begin{equation*}
    A^3(\sigma)\cong A^3(\sigma)\otimes\chi.
\end{equation*}
Now, from \eqref{eq:dec:A3:E} and the fact that $\sym^2(\sigma_E)\cong\sym^2(\sigma_E)\otimes\mu$ by Theorem \ref{thm:A3} we have that
\begin{equation*}
 \begin{split}
    \sigma_E\otimes A^3(\sigma_E) & \cong \left(\sigma_E\otimes \sigma_E\otimes\mu \right)\oplus \left(\sigma_E\otimes \sigma_E\otimes\mu^2\right)\\
    & \cong \left(\sym^2(\sigma_E)\oplus \omega_{\sigma_E}\mu \right)\oplus \left(\sym^2(\sigma_E)\oplus \omega_{\sigma_E}\mu^2\right).
\end{split}
\end{equation*}
By replacing the above in \eqref{eq:sigma_x_A3:E}, we have 
\begin{equation*}
    \tau_E\cong \omega_{\sigma_E}\mu \oplus \omega_{\sigma_E}\mu^2.
\end{equation*}
Therefore, $\tau\cong{{\rm{Ind}}}_E^F(\omega_{\sigma_E}\mu)$ and we obtain the decomposition
\begin{equation*}
    A^4(\sigma)\cong {{\rm{Ind}}}_E^F (\omega_{\sigma_E}\mu) \oplus \sym^2(\sigma)\otimes\chi.
\end{equation*}
We are done. 
\end{proof}

Analogous to the dihedral and tetrahedral cases, Corollaries \ref{cor:sym:dihedral:reducible} and \ref{cor:sym:tetrahedral:reducible} respectively, for octahedral representations we have the following result.

\begin{cor}\label{cor:sym:octahedral:reducible}
The representation $\sigma$ is octahedral if and only if $M=3$.
\end{cor}

Given our results thus far, the following criterion for the symmetric fifth power is quickly obtained.

\begin{thm}\label{thm:sym45}
The following are equivalent:
\begin{itemize}
    \item[(i)] $\sym^5(\sigma)$ is reducible.
    \item[(ii)] $\sym^4(\sigma)$ is reducible.
    \item[(iii)] $\sigma$ is either dihedral, tetrahedral or octahedral.
\end{itemize}
\end{thm}
\begin{proof}
The equivalence of (i) and (ii) follows from Lemmas \ref{lem:redsymm:geqn} and \ref{lem:red:symm}. And, from our prior results on $\sym^n(\sigma)$ for $n = 4,3,2$, Theorems \ref{thm:sym4}, \ref{thm:A3} and \ref{thm:sym2} respectively, we can infer that $\sym^4(\sigma)$ is reducible if and only if $\sigma$ is either dihedral, tetrahedral or octahedral.
\end{proof}

\section{Reducibility Criteria for Symmetric Powers of \\ Galois Representations: The Icosahedral Case}\label{sec:Galois:crit:III}

The main aim of this section is to establish the remaining reducibility criteria involving the symmetric sixth, proving that if $\sym^6(\sigma)$ is reducible then ${\rm proj}(\sigma(G_F))$ is finite. Of particular interest are icosahedral represenations. We keep the notation of \S\S~\ref{sec:Galois:crit:I} and \ref{sec:Galois:crit:II}, where in particular all of our representations are $\ell$-adic.

\subsection{}\label{ss:sym6:crit:I}
In analogy to the previous cases, arguing with Corollaries \ref{cor:sym:dihedral:reducible}, \ref{cor:sym:tetrahedral:reducible}, and now including Corollary \ref{cor:sym:octahedral:reducible}, it is enough to study the reducibility of $\sym^6(\sigma)$ when $\sym^4(\sigma)$ is irreducible.

\begin{thm}\label{thm:sym6:dec}
If $\sym^4(\sigma)$ is irreducible, then the following are equivalent.
\begin{enumerate}
    \item [\namedlabel{item:sym6:i}{(i)}] $\sym^6(\sigma)$ is reducible.
    \item [\namedlabel{item:sym6:ii}{(ii)}] There exists an irreducible two-dimensional representation $\sigma'$ of $G_F$ such that $\sym^5(\sigma)\cong {\rm{Ad}}(\sigma)\otimes\sigma'$.
\end{enumerate}
The representation $\sigma'$ is uniquely determined up to isomorphism by \ref{item:sym6:ii}.
\end{thm}

\begin{proof}
Let us suppose first that $\sym^5(\sigma)\cong {\rm{Ad}}(\sigma)\otimes\sigma'$ for some $\sigma'$. Then
\begin{equation}\label{eq:sigma_x_sym5:i}
    \sigma\otimes\sym^5(\sigma)\cong \sigma\otimes{\rm{Ad}}(\sigma)\otimes \sigma'\cong \left( A^3(\sigma)\otimes\sigma'\right)\oplus \left( \sigma\otimes\sigma'\right).
\end{equation}
On the other hand, we have the identity
\begin{equation}\label{eq:sigma_x_sym5:ii}
    \sigma\otimes \sym^5(\sigma)\cong \sym^6(\sigma)\oplus \sym^4(\sigma)\otimes \omega_\sigma.
\end{equation}
By comparing the decompositions \eqref{eq:sigma_x_sym5:i} and \eqref{eq:sigma_x_sym5:ii}, and noting that $\sym^4(\sigma)\otimes \omega_\sigma$ is irreducible, we see that $\sym^6(\sigma)$ is reducible.

Let us suppose now that $\sym^6(\sigma)$ is reducible. Let $\tau$ be an irreducible component of $\sym^6(\sigma)$ of dimension $r\leq 3$, which exists by the semisimplicity of $\sym^6(\sigma)$. By \eqref{eq:sigma_x_sym5:ii}, $\tau$ is a subrepresentation of $\sigma\otimes\sym^5(\sigma)$ and then $\sym^5(\sigma)$ is a subrepresentation of $\sigma^\vee\otimes \tau$, by \ref{item:app:3.2}. Since $\sym^5(\sigma)$ is $6$-dimensional, we necessarily have that $r=3$ and 
\begin{equation}\label{eq:sym^5:dec:A6}
    \sym^5(\sigma)\cong \sigma^{\vee}\otimes \tau.
\end{equation}
We observe that $\sym^6(\sigma)\cong \Pi \oplus \tau$ for some irreducible 4-dimensional representation $\Pi$, since $\sym^6(\sigma)$ cannot have irreducible summands of dimension less than 3 by the above. Now, $\sym^5(\sigma)$ is a subrepresentation of $\sigma^\vee\otimes\Pi$, and by comparing dimensions we deduce that there exists an irreducible 2-dimensional representation $\sigma'$ such that 
\begin{equation}\label{eq:sigma'}
    \sigma^\vee\otimes\Pi\cong \sym^5(\sigma)\oplus \sigma'.
\end{equation}
Using \ref{item:app:3.2} again and comparing dimensions, we obtain that 
\begin{equation*}
    \Pi\cong \sigma\otimes \sigma'.
\end{equation*}
Then
\begin{equation*}
    \sigma^{\vee}\otimes \Pi\cong \sigma^\vee\otimes(\sigma\otimes\sigma')\cong\left({\rm{Ad}}(\sigma)\otimes\sigma'\right) \oplus \sigma'.
\end{equation*}
By comparing with \eqref{eq:sigma'} we have that $\sym^5(\sigma)\cong {\rm{Ad}}(\sigma)\otimes\sigma'$, i.e., \ref{item:sym6:ii} holds.

Finally, let us see that the representation $\sigma'$ is uniquely determined by \ref{item:sym6:ii}. Let $\sigma''$ be another (necessarily irreducible) representation such that $\sym^5(\sigma)\cong {\rm Ad}(\sigma)\otimes\sigma''$. Then 
\begin{equation*}
	\sigma\otimes\sym^5(\sigma)\cong \left( A^3(\sigma)\otimes\sigma'\right)\oplus \left( \sigma\otimes\sigma'\right)\cong\left( A^3(\sigma)\otimes\sigma''\right)\oplus \left( \sigma\otimes\sigma''\right).
\end{equation*}
The above being in turn isomorphic to 
\begin{equation*}
\sym^6(\sigma)\oplus \sym^4(\sigma)\otimes \omega_\sigma\cong \Pi \oplus \tau \oplus \sym^4(\sigma)\otimes \omega_\sigma,	
\end{equation*}
by comparing dimensions, we obtain that $\Pi\cong\sigma\otimes\sigma''$. Thus, $\sigma''$ is a subrepresentation of $\sigma^\vee\otimes\Pi$. By \eqref{eq:sigma'}, we see that $\sigma''\cong\sigma'$, as wanted. 
\end{proof}

\begin{prop}\label{prop:sym6:chi}
Assume that $\sym^4(\sigma)$ is irreducible and $\sym^6(\sigma)$ is reducible. Let $\sigma'$ be the representation of Theorem \ref{thm:sym6:dec}. Then there exists a character $\chi$ such that
\begin{equation*}\label{eq:prop:sym6}
\sym^5(\sigma)\cong {\rm Ad}(\sigma')\otimes\sigma\otimes\chi. 	
\end{equation*}
In this case, we have the decomposition 
\begin{equation*}\label{eq:prop:sym6:dec}
    \sym^6(\sigma)\cong \sigma\otimes\sigma' \oplus {\rm{Ad}}(\sigma')\otimes \chi \omega_\sigma.
\end{equation*}	
\end{prop}

\begin{proof}

From \eqref{eq:sym^5:dec:A6} we obtain
\begin{equation*}
    \sigma\otimes\sym^5(\sigma)\cong\sigma\otimes \sigma^{\vee}\otimes\tau\cong {\rm{Ad}}(\sigma)\otimes\tau\oplus \tau.
\end{equation*}
On the other hand, we have that 
\begin{equation*}
    \sigma\otimes\sym^5(\sigma)\cong\sym^6(\sigma)\oplus \sym^4(\sigma)\otimes \omega_\sigma \cong \sigma\otimes\sigma' \oplus \tau \oplus \sym^4(\sigma)\otimes \omega_\sigma.
\end{equation*}
Thus, 
\begin{equation*}
    {\rm{Ad}}(\sigma)\otimes\tau \cong \sigma\otimes\sigma' \oplus \sym^4(\sigma)\otimes \omega_\sigma.
\end{equation*}
Therefore, $\sigma\otimes\sigma'$ is a subrepresentation of ${\rm Ad}(\sigma)\otimes\tau$. By \ref{item:app:3.2}, noting that ${\rm Ad}(\sigma)$ is selfdual, we have that $\tau$ is a subrepresentation of 
\begin{equation*}
    {\rm Ad}(\sigma)\otimes\sigma\otimes\sigma'\cong \left( A^3(\sigma)\otimes \sigma'\right) \oplus \left( \sigma\otimes \sigma'\right).
\end{equation*}
We note that $\tau$ is not a subrepresentation of $\sigma\otimes\sigma'$, since $\sigma\otimes\sigma'$ is irreducible. Thus, we have that $\tau$ is a subrepresentation of $A^3(\sigma)\otimes\sigma'$. Then $A^3(\sigma)$ is a subrepresentation of $\sigma'^\vee\otimes\tau$, and there exists a (necessarily irreducible) two-dimensional representation $\tau_0$ such that
\[
\sigma'^\vee\otimes\tau\cong A^3(\sigma)\oplus \tau_0.
\] 
Then $\tau$ is a subrepresentation of $\sigma'\otimes\tau_0$ and, by comparing dimensions, we see that $\sigma'\otimes\tau_0$ contains a character. Then, $\tau_0\cong \sigma'\otimes \mu$, for some character $\mu$, and 
\begin{equation*}
    \sigma'\otimes\tau_0\cong \sym^2(\sigma')\otimes\mu\oplus \omega_{\sigma'}\mu.
\end{equation*}
The above implies that $\tau\cong \sym^2(\sigma')\otimes\mu$. We then have that 
\begin{equation*}
    \sym^5(\sigma)\cong \sigma^{\vee}\otimes \tau\cong {\rm{Ad}}(\sigma')\otimes \sigma\otimes \chi,
\end{equation*}
where $\chi=\omega_{\sigma'}\omega_{\sigma}^{-1}\mu$. As for the decomposition, we note that 
\begin{equation*}
	\sym^6(\sigma)\cong \Pi \oplus \tau \cong \sigma\otimes\sigma' \oplus \sym^2(\sigma')\otimes\mu\cong \sigma\otimes\sigma' \oplus {\rm{Ad}}(\sigma')\otimes \chi \omega_\sigma.
\end{equation*}
\end{proof}

\subsection{Remarks}\label{rmks:quasi-ico:unique:chi} Over number fields, Ramakrishnan refers to the representations satisfying the equivalente conditions of Theorem \ref{thm:sym6:dec} as quasi-icosahedral, see \cite{Ra2009}. Over function fields, we shall make a refinement to the irreducibility criterion presented in Theorem \ref{thm:sym6:dec}. We prove that these representations are icosahedral in Theorem~\ref{thm:icosahedral}, by proving they have finite image in ${\rm PGL}_2(\overline{\mathbb{Q}}_\ell)$.

Furthermore, observe that the representation $\sigma'$ of Theorem \ref{thm:sym6:dec} is very peculiar. For instance, $\sigma'\not\cong \sigma\otimes \mu$ for every character $\mu$. Also, since $\sym^5(\sigma)\cong {\rm{Ad}}(\sigma')\otimes\sigma\otimes\chi$ is irreducible, then ${\rm{Ad}}(\sigma')\not\cong {\rm{Ad}}(\sigma)\otimes\mu$ for every character $\mu$. However, we shall prove in Proposition \ref{prop:gal:sym6:sym2} below that $\sym^3(\sigma')\cong \sym^3(\sigma)\otimes \mu$ for some character $\mu$. This identity brings to light the fact that the properties of the character $\chi$ of Proposition \ref{prop:sym6:chi} determine it uniquely, see Corollary~\ref{cor:unique:chi}.

\subsection{} We now address a couple of delicate aspects of the criteria, giving a precise description of the representation $\sigma'$ and character $\chi$ appearing in \S~\ref{ss:sym6:crit:I}.

\begin{prop}\label{prop:gal:sym6:sym2}
Let us assume that $\sym^4(\sigma)$ is irreducible and that $\sym^6(\sigma)$ is reducible. Let $\sigma'$ be the representation in Theorem \ref{thm:sym6:dec}. Then there exists a unique character $\mu$ such that
\begin{equation*}
    \sym^3(\sigma')\cong \sym^3(\sigma)\otimes \mu.
\end{equation*}
More precisely, $\mu=\omega_{\sigma'} \omega_\sigma^{-1} \chi$, where $\chi$ is as in Proposition \ref{prop:sym6:chi}, and satisfies
\begin{equation*}
	\mu=(\eta \omega_\sigma^2)^3,
\end{equation*}
for a quadratic character $\eta$.
\end{prop}

\begin{proof}
The unicity of $\mu$ follows from the irreducibility of $\sym^4(\sigma)$ and Theorem \ref{thm:sym4}. Let $\chi$ be a character as in Proposition \ref{prop:sym6:chi}. We recall that we have
\begin{equation*}
    \sym^5(\sigma)\cong {\rm{Ad}}(\sigma')\otimes\sigma\otimes \chi\cong {\rm{Ad}}(\sigma)\otimes \sigma'.
\end{equation*}
Then 
\begin{equation*}
    \sigma\otimes \sym^5(\sigma) \cong \sigma \otimes {\rm{Ad}}(\sigma')\otimes\sigma\otimes \chi \cong \left( \sym^2(\sigma)\otimes {\rm{Ad}}(\sigma')\otimes\chi\right) \oplus \left( {\rm{Ad}}(\sigma')\otimes\chi \omega_\sigma\right),
\end{equation*}
and 
\begin{equation*}
    \sigma\otimes \sym^5(\sigma) \cong \sigma \otimes {\rm{Ad}}(\sigma)\otimes \sigma'\cong \left( A^3(\sigma)\otimes \sigma'\right) \oplus \left( \sigma\otimes\sigma'\right).
\end{equation*}
Since $\sigma\otimes\sigma'$ is irreducible, the identities above imply that it is a subrepresentation of $\sym^2(\sigma)\otimes {\rm{Ad}}(\sigma')\otimes\chi$. Then $\chi$ is a character contained in $\sym^2(\sigma^\vee)\otimes {\rm{Ad}}(\sigma')\otimes \sigma\otimes\sigma'$. Using Clebsch-Gordan, we see that $\chi$ is contained in 
\begin{equation*}
 \left(A^3(\sigma)\otimes A^3(\sigma')\otimes \omega_\sigma^{-1}\right)\oplus \left( A^3(\sigma)\otimes \sigma'\otimes \omega_\sigma^{-1}\right) \oplus  \left(\sigma^\vee\otimes A^3(\sigma') \right) \oplus \left(\sigma^\vee\otimes\sigma'\right).
\end{equation*}
The representation $\sigma\otimes\sigma'$ being irreducible, the character $\chi$ is not contained in $\sigma^\vee\otimes\sigma'$. Also, since $A^3(\sigma)$ and $\sigma'$ are irreducible and of different dimension, $\chi$ cannot be contained in $A^3(\sigma)\otimes \sigma'\otimes \omega_\sigma^{-1}$. Let us assume that $\sigma^\vee\otimes A^3(\sigma')$ contains $\chi$.  Then $A^3(\sigma')$ is reducible and $\sigma\otimes\chi$ is one of its irreducible factors. Now, since ${\rm{Ad}}(\sigma')$ is irreducible (otherwise, $\sym^5(\sigma)\cong {\rm{Ad}}(\sigma')\otimes\sigma\otimes\chi$ would be reducible), then by Theorem \ref{thm:A3} we have that 
\begin{equation*}
    A^3(\sigma')\cong \sigma'\otimes \mu_0 \oplus \sigma'\otimes \mu_0^2,
\end{equation*}
for some character $\mu_0$. Then $\sigma\otimes \chi$ is (isomorphic to) either $\sigma'\otimes \mu_0$ or $\sigma'\otimes \mu_0^2$. This is a contradiction, for we know that $\sigma'\not\cong \sigma \otimes \mu$ for every character $\mu$.
\noindent 

Therefore, $A^3(\sigma)\otimes A^3(\sigma')\otimes \omega_\sigma^{-1}$ contains $\chi$. This implies that $A^3(\sigma')$ is irreducible and $A^3(\sigma')\cong A^3(\sigma)^\vee\otimes \omega_\sigma \chi$. Then if we let $\mu = \omega_{\sigma'} \omega_\sigma^{-1} \chi$, we get
\begin{equation*}
 \sym^3(\sigma')\cong \sym^3(\sigma)\otimes \mu,
\end{equation*}
as desired. 

Now, by comparing determinants from the relation ${\rm{Sym}}^5(\sigma)\cong {\rm{Ad}}(\sigma)\otimes \sigma'$, we have
	\begin{equation}\label{eq:omega_sgm15}
		\omega_\sigma^{15}=\omega_{\sigma'}^{3}.
	\end{equation}
On the other hand, by taking contragradient from both sides of ${\rm{Sym}}^3(\sigma')\cong {\rm{Sym}}^3(\sigma)\otimes \mu$, we obtain ${\rm{Sym}}^3(\sigma')\otimes \omega_{\sigma'}^{-3}\cong {\rm{Sym}}^3(\sigma)\otimes \omega_\sigma^{-3}\mu^{-1}$. Thus, 
	\begin{equation*}
		{\rm{Sym}}^3(\sigma)\otimes \mu \cong {\rm{Sym}}^3(\sigma')\cong {\rm{Sym}}^3(\sigma)\otimes (\omega_{\sigma'}\omega_\sigma^{-1})^3\mu^{-1},
	\end{equation*}
i.e.
\begin{equation*}
{\rm{Sym}}^3(\sigma)\cong {\rm{Sym}}^3(\sigma)\otimes (\omega_{\sigma'}\omega_\sigma^{-1})^3\mu^{-2}.
\end{equation*}

Since $\sigma$ is not of octahedral type, the above implies that $\mu^2=(\omega_{\sigma'}\omega_\sigma^{-1})^3$. But from \eqref{eq:omega_sgm15} we have $(\omega_{\sigma'}\omega_\sigma^{-1})^3=\omega_\sigma^{12}$. Thus 
	\begin{equation*}
		\mu=\eta \omega_\sigma^6 = (\eta \omega_\sigma^2)^3
	\end{equation*}
	for some quadratic character $\eta$. 
\end{proof}

We now register a consequence that we remarked in \S~\ref{rmks:quasi-ico:unique:chi}.

\begin{cor}\label{cor:unique:chi}
Let $\sigma'$ be the irreducible two-dimensional Galois representation of Theorem \ref{thm:sym6:dec} satisfying
\[ \sym^5(\sigma)\cong {\rm{Ad}}(\sigma)\otimes\sigma'. \]
Then the character $\chi$ of Proposition~\ref{prop:sym6:chi} satisfying
\[
    \sym^5(\sigma)\cong {\rm Ad}(\sigma')\otimes\sigma\otimes\chi
\]
is unique with respect to this property.
\end{cor}

\begin{proof}
If $\xi$ were another character satisfying $\sym^5(\sigma)\cong {\rm Ad} (\sigma') \otimes \sigma \otimes \xi$, then Proposition~\ref{prop:sym6:chi} tells us that
\[ \sym^6(\sigma)\cong \sigma\otimes\sigma' \oplus {\rm{Ad}}(\sigma')\otimes \xi \omega_\sigma. \]
Hence, we would have
\[ {\rm{Ad}}(\sigma')\cong {\rm Ad}(\sigma')\otimes \xi\chi^{-1}. \]
If $\xi\chi^{-1}$ were not trivial, then by Theorem \ref{thm:A3}, $\sym^3(\sigma')$ would be reducible, a contradiction.
\end{proof}

\subsection{}
Let us denote $\xi:=\eta^{-1} \omega_\sigma^{-2}$ and $\widetilde\sigma:=\sigma'\otimes\xi$. Thus, $\sym^3(\sigma)\cong\sym^3(\widetilde\sigma)$. We observe that this implies in particular that $\sym^3(\widetilde\sigma)$ is irreducible. Also, by Theorem \ref{thm:sym4}, we have that $\sym^4(\widetilde\sigma)$ is irreducible. Furthermore, by decomposing via the Clebsch-Gordan formulas both sides of
\[ \sym^3(\sigma)\otimes \sym^3(\sigma)\cong \sym^3(\widetilde\sigma)\otimes\sym^3(\widetilde\sigma)\]
and using the fact that $\sym^6(\sigma)$
is reducible, we deduce that $\sym^6(\widetilde\sigma)$ is reducible. As a conclusion of this, $\widetilde\sigma$ satisfies the conditions in Theorem \ref{thm:sym6:dec}. 

 We write $\Gamma = \sigma(G_F)$ and $\widetilde \Gamma = \widetilde\sigma(G_F)$. We shall need the following useful fact. 

\begin{lem}
The groups $\Gamma$ and $\widetilde \Gamma$ consist of semisimple automorphisms. 
\end{lem}

\begin{proof}
Let $g\in \Gamma$, and let us assume that $g$ is not semisimple. Then, the Jordan decomposition tells us that there is a suitable basis for $V$ where we may write
\begin{equation*}
	g=\left(\begin{array}{cc}
		\lambda & 1\\
		0 & \lambda
	\end{array}\right)
\end{equation*}
for some $\lambda\in \overline{\mathbb{Q}}_\ell^\times$. But then the Jordan normal form of $\sym^6(g)$ has a single block, i.e., none of its invariant subspaces have a complement. This is not possible, since $\sym^6(\sigma)$ is semisimple and reducible. Since $\sym^6(\widetilde\sigma)$ is also reducible, the same applies to $\widetilde \Gamma$.
\end{proof}

In our study of the image $\Gamma = \sigma(G_F)$, we come across the following \emph{no-small finite subgroups} property:
\begin{quote}
    A compact subgroup $K$ of ${\rm GL}_2(\overline{\mathbb{Q}}_\ell)$, has itself a compact open subgroup $\mathcal{U}$ without non-trivial finite subgroups.
\end{quote}
Actually, if $K$ is any compact subgroup of ${\rm GL}_2(\overline{\mathbb{Q}}_\ell)$, then there exists a finite extension of $\mathbb{Q}_\ell$ with ring of integers $\mathcal{O}$ such that $K\subseteq {\rm GL}_2(\mathcal{O})$. It is well known that ${\rm GL}_2(\mathcal{O})$ contains a compact open subgroup, say $\mathcal{U}_0$, without non-trivial finite subgroups. Then it suffices to take $\mathcal{U}=\mathcal{U}_0\cap K$.

We observe that $\mathcal{U}$ could very well be trivial, and clearly this is the case when $K$ is finite. Note that $\Gamma 
 = \sigma(G_F)$ is compact, hence it satisfies the no-small finite subgroups property.

\begin{thm}\label{thm:icosahedral}
	If ${\rm{Sym}}^4(\sigma)$ is irreducible and ${\rm{Sym}}^6(\sigma)$ is reducible, then $\sigma$ is icosahedral. 
\end{thm}

Our proof relies on the following technical lemma. That $\sigma$ is icosahedral follows once we settle that  $J = {\rm proj}(\sigma(G_F))$ is finite, for we are assuming that $\sigma$ is neither dihedral, tetrahedral, nor octahedral.

\begin{lem}\label{lem:icosahedral:iso}
Assume that $\,\mathcal{U}_{\,\Gamma}$ is a non-trivial compact open subgroup of $\,\Gamma=\sigma(G_F)$ without non-trivial finite subgroups. Let
\[ \varphi \colon {\rm{Sym}}^3(\sigma) \xrightarrow {\ \sim\ }{\rm{Sym}}^3(\widetilde\sigma) \]
be a given $G_F$-isomorphism. Then there exist bases for $V$ and $\widetilde V$ and $g\in{\rm{M}}_2(\overline{\mathbb{Q}}_\ell)$ such that the corresponding matrix for $\varphi$ is ${\rm{Sym}}^3(g)$. 
\end{lem}

\begin{proof}[Proof \emph{(of Lemma \ref{lem:icosahedral:iso})}]
Let us first assume that $\omega_\sigma$ is of finite order. Let $h \in \mathcal{U}_{\,\Gamma}$ be different from the identity. We choose a basis for $V$ such that $h$ is representable by a diagonal matrix, say $h=\diag (a,b)$. Let us choose a basis for $\widetilde V$ such that 
\begin{equation*}
	\varphi\cdot{\rm{Sym}}^3(h)\cdot\varphi^{-1}=\varphi\cdot \diag(a^3,a^2b,ab^2,b^3)\cdot \varphi^{-1}\in {\rm{Aut}}_{\overline{\mathbb{Q}}_\ell}(\widetilde V)
\end{equation*}
is representable by a diagonal matrix. Since $\varphi \colon {\rm{Sym}}^3(\sigma) \rightarrow{\rm{Sym}}^3(\widetilde\sigma)$ is a $G_F$-isomorphism, we may write 
\begin{equation}\label{eq:varphi_sym3(h)}
	\varphi\cdot \diag(a^3,a^2b,ab^2,b^3)\cdot \varphi^{-1}=  \diag(\alpha^3,\alpha^2\beta,\alpha\beta^2,\beta^3)
\end{equation}
for some $\alpha,\beta\in\overline{\mathbb{Q}}_\ell$. We note that all the non-zero entries of ${\rm{Sym}}^3(h)$ are different from each other. Otherwise we would have that $a^k=b^k$ for some integer $k>0$, and since $h$ has a finite-order determinant, then $a^n=b^n=1$ for some other integer $n>0$. But then $h$ would generate a finite non-trivial group, contradicting the no-small finite subgroups property of $\Gamma$. Thus, we may write $\varphi$ as $TM$, where $T$ is diagonal and $M$ a permutation matrix. Let us see actually that $M$ is either $I=(\delta_{i,j})$ or $I^{-}=(\delta_{i,5-j})$. From \eqref{eq:varphi_sym3(h)}, we have 
\begin{equation*}
	\begin{split}
		\alpha^3 & =a^{n_0}b^{3-n_0},\\
		\alpha^2\beta & = a^{n_1}b^{3-n_1},\\
		\alpha\beta^2 &= a^{n_2}b^{3-n_2},\\
		\beta^3 & = a^{n_3}b^{3-n_3},
	\end{split}
\end{equation*}
for some permutation $(n_0,n_1,n_2,n_3)$ of $(0,1,2,3)$. Again, since $a^k\not= b^k$ for all integers $k\not= 0$, the equations above imply 
\begin{equation*}
	\begin{split}
		n_0+n_3 & =n_1+n_2, \\
		n_0+n_2 & = 2n_1.
	\end{split}
\end{equation*}
It is readily seen that the unique solutions of the latter are $(0,1,2,3)$ and $(3,2,1,0)$, as claimed. 

Since both $I$ and $I^-$ are of the form ${\rm{Sym}}^3(g_0)$ for some $g_0$, it suffices to prove that $T={\rm{Sym}}^3(g)$ for some $g$. Let us suppose that
\begin{equation*}
    T=\diag (t_1,t_2,t_3,t_4).
\end{equation*}
Necessary and sufficient conditions for $T$ to be of the form ${\rm{Sym}}^3(g)$ for some $g$ are $t_1t_4=t_2t_3$ and $t_1t_3=t_2^2$. Let us verify these. Since $\sigma$ is not dihedral, there exists
	$\left( \begin{array}{cc}
		x & y \\
		w & z\\ 
		
	\end{array}\right)\in J$ such that $xy\not = 0$ or $wz\not = 0$. We note that $T\cdot {\rm{Sym}}^3\left( \left( \begin{array}{cc}
		x & y \\
		w & z\\ 
	\end{array}\right)\right) \cdot T^{-1}$ is equal to
	\begin{equation}\label{eq:TsymT^-1}
	\left( \begin{array}{cccc}
		x^3 & x^2y\cdot t_1t_2^{-1}& xy^2\cdot t_1t_3^{-1}& y^3\cdot t_1t_4^{-1} \\
		 * & * & * & *\\
		 * & * & * & * \\ 
		 
		w^3\cdot t_4t_1^{-1} & w^2z\cdot t_4t_2^{-1}& wz^2\cdot t_4t_3^{-1}& z^3 
	\end{array}\right),
	\end{equation}
 where we will only use the entries in the first and last row of the previous matrix. 
 
 Let us suppose that $xy\not = 0$. Since the matrix in \eqref{eq:TsymT^-1} is ${\rm{Sym}}^3(\widetilde h)$ for some $\widetilde h$, then we certainly have
 \begin{equation*}
    \begin{split}
 	x^3y^3\cdot t_1t_4^{-1} & =(x^2y\cdot t_1t_2^{-1})(xy^2\cdot t_1t_3^{-1}),\\
 	x^3(xy^2\cdot t_1t_3^{-1}) & =(x^2y\cdot t_1t_2^{-1})^2.
 	\end{split}
    \end{equation*}
Since $xy\not = 0$, then $t_1t_4^{-1}=t_1^2t_2^{-1}t_3^{-1}$ and $t_1t_3^{-1}=t_1^2t_2^{-2}$. It is straightforward to see that these lead us to the desired properties $t_1t_4=t_2t_3$ and $t_1t_3=t_2^2$. If $wz\not=0$, then we must pay attention to the fourth row of \eqref{eq:TsymT^-1}, and the above works similar.
	
In general, we note that, by \cite{HeLe2011} \S\,IV.2.9, there exists an $\ell$-adic character $\chi$ such that $\det (\sigma\otimes\chi)$ is of finite order. Since every $G_F$-isomorphism
\[ \varphi \colon \rm{Sym}^3(\sigma\otimes\chi) \xrightarrow{\ \sim\ }{\rm{Sym}}^3(\widetilde\sigma\otimes\chi) \]
is in turn a $G_F$-isomorphism
\[ \varphi \colon {\rm{Sym}}^3(\sigma)\xrightarrow{\ \sim\ }{\rm{Sym}}^3(\widetilde\sigma), \]
the proposition follows.
\end{proof}

\begin{proof}[Proof of Theorem \ref{thm:icosahedral}]
Let $\mathcal{U}_{\,\Gamma}$ be a compact open subgroup of $\Gamma$ without non-trivial finite subgroups. If $\mathcal{U}_{\,\Gamma}$ is trivial, then $\Gamma$ is finite and we are done. Hence, we may assume that $\mathcal{U}_{\,\Gamma}$ is non-trivial.

We fix bases for $V$ and $\widetilde V$ and choose $g\in {\rm{Hom}}_{\overline{\mathbb{Q}}_\ell}(V,\widetilde V)$ such that ${\rm{Sym}}^3(g)$ is a $G_F$-isomorphism ${\rm{Sym}}^3(\sigma)\xrightarrow{\ \sim\ }{\rm{Sym}}^3(\widetilde\sigma)$, just as in Lemma \ref{lem:icosahedral:iso}. Since the group homomorphism ${\rm{Sym}}^3\colon{\rm{GL}}_2(\overline{\mathbb{Q}}_\ell)\to {\rm{GL}}_4(\overline{\mathbb{Q}}_\ell)$ has finite kernel, then there exists an open subgroup $U$ of $\Gamma$ such that ${\rm{Sym}}^3\vert_U$ is injective. 
	
If we set $H=\sigma^{-1}(U)$, then $H$ is an open subgroup of $G_F$ and $g$ becomes an $H$-isomorphism $V\xrightarrow{\ \sim\ }\widetilde V$. We observe that since $\sigma_H\cong \widetilde\sigma_H\cong \sigma'_H\otimes\xi_H$, and 
	\begin{equation*}
		{\rm{Sym}}^5(\sigma_H)\cong {\rm{Ad}}(\sigma_H)\otimes \sigma'_H,
	\end{equation*}
	then ${\rm{Sym}}^5(\sigma_H)$ is reducible. This implies that ${\rm proj}(\sigma(H))$ is finite. Since $H$ is open in $G_F$, we are done.
\end{proof}

\subsection{Higher Symmetric Powers}

We here register Ramakrishnan's irreducibility criterion for higher symmetric powers, Theorem A' (c) of \cite{Ra2009}.

\begin{thm}\label{thm:higher}
Let $n>6$ be an integer. Then $\sym^n(\sigma)$ is irreducible if and only if $\sym^6(\sigma)$ is irreducible. 
\end{thm}

The proof for complex representations over number fields of [\emph{loc.\,cit.}], is also valid in the case of $\ell$-adic representations over function fields. It is a formal consequence of Lemmas \ref{lem:redsymm:geqn} and \ref{lem:red:symm}.

\subsection{}
The following theorem summarizes what we have done so far.

\begin{thm}\label{thm:symm:Galois}
Let $\sigma\colon G_F\to {\rm GL}(V)$ be an irreducible 2-dimensional $\ell$-adic Galois representation. Then the following are equivalent:
\begin{itemize}
    \item[(i)] ${\rm proj}(\sigma(G_F))$ is finite.
    \item[(ii)] There exists an integer $n\geq 2$ such that $\sym^n(\sigma)$ is reducible.
    \item[(iii)] $\sym^6(\sigma)$ is reducible.
\end{itemize}
To be more precise, if the above equivalent conditions are met we then have the following classification. 
\begin{itemize}
    \item[] M=1: $\sigma$ is dihedral.
	\item[] M=2: $\sigma$ is tetrahedral.
	\item[] M=3: $\sigma$ is octahedral. 
	\item[] M=5: $\sigma$ is icosahedral.
\end{itemize}
And, $\sym^5(\sigma)$ is irreducible if and only if $\sym^4(\sigma)$ is irreducible.
\end{thm}

\begin{proof}
At this point, only the equivalence of (i), (ii) and (iii) require comment.

Assume that ${\rm proj}(\sigma(G_F))$ is finite, in such a way that $\Gamma$ is finite modulo the center. Now, for each $n$, $\sym^n(\sigma)$ can be seen as a representation of $\Gamma$ over the algebraically closed field of characteristic zero $\overline{\mathbb{Q}}_\ell$. Notice that $\Gamma$ can only have irreducible representations (necessarily finite-dimensional) of finitely many dimensions, up to isomorphism. Necessarily, $\sym^n(\sigma)$ is reducible for some integer $n \geq 2$ and (ii) holds. 

Now, assume that $\sym^n(\sigma)$ is reducible for some $n\geq 2$. If $n\leq 6$, then $\sym^6(\sigma)$ is reducible by Lemma~\ref{lem:redsymm:geqn}. And, if $n>6$ Theorem \ref{thm:higher} implies that $\sym^6(\sigma)$ is reducible.

Finally, assume that $\sym^6(\sigma)$ is reducible. Then the maximal irreducible symmetric power satisfies $M\leq 5$. The case of $M=4$ cannot happen by Lemma~\ref{lem:red:symm}, and each case among $M=1,2,3$ is covered by one of the irreducibility criteria of \S~\ref{sec:Galois:crit:II} and Theorem \ref{thm:icosahedral} applies to $M=5$. It follows that ${\rm proj}(\sigma(G_F))$ is finite.
\end{proof}

\section{Passage to Automorphic Representations}\label{pass:to:auto}

The landmark results of L. Lafforgue over a global function field $F$ \cite{LLaff2002}, establish a correspondence between cuspidal automorphic representations of ${\rm GL}_n(\mathbb{A}_F)$ with finite central character and irreducible $n$-dimensional $\ell$-adic Galois representations that are unramified almost everywhere and have finite determinant.

In \cite{HeLe2011}, the global Langlands correspondence is slightly expanded. It is phrased in the context of the more general $\ell$-adic Weil representations, which we have taken to be unramified almost everywhere by definition in \S\,\ref{prelim}. The crucial point in this generalization is that for an arbitrary $\ell$-adic Weil representation $\sigma\colon \mathcal{W}_F\to{\rm GL}(V)$, there exists an $\ell$-adic character $\chi$ such that $\sigma\otimes\chi$ can be extended to an $\ell$-adic Galois representation with finite-order determinant. That is to say, $\sigma\otimes\chi$ lies in the setting of L. Lafforgue \cite{LLaff2002}. More precisely, in \cite{HeLe2011} the authors establish a bijection between the set $\mathcal{G}^n_\ell(F)$ of irreducible $n$-dimensional $\ell$-adic representations of $\mathcal{W}_F$ and the set $\mathcal{A}^n(F)$ of cuspidal automorphic representations of ${\rm GL}_n(\mathbb{A}_F)$, such that if $\sigma\in\mathcal{G}^n_\ell(F)$ corresponds to $\pi\in\mathcal{A}^n(F)$, then it agrees at every place with the local Langlands correspondence established by Laumon, Rapoport and Stuhler \cite{LaRaSt1993}. We write $\sigma \longleftrightarrow \pi$ to denote corresponding representations.

An $\ell$-adic Galois character $\chi$ corresponds to an automorphic character which we again denote in this case by $\chi$, this is possible via class field theory after fixing a field isomorphism $\iota\colon \overline{\mathbb{Q}}_\ell\to \mathbb{C}$, cf. \S~IV.2 of \cite{HeLe2011}. Furthermore, if $\pi\in\mathcal{A}^n(F)$ corresponds to $\sigma\in\mathcal{G}^n_\ell(F)$, then the central character $\omega_\pi$ of $\pi$ corresponds to $\omega_\sigma = \det(\sigma)$.

The correspondence is further extended so that a semisimple $n$-dimensional $\ell$-adic representation $\sigma$ of $\mathcal{W}_F$, written as a direct sum
\[ \sigma = \sigma_1 \oplus \cdots \oplus \sigma_d, \quad \sigma_i \in \mathcal{G}^{n_i}_\ell(F), \]
corresponds to an automorphic representation $\pi$ of ${\rm GL}_n(\mathbb{A}_F)$, given by an isobaric sum
\[ \pi = \pi_1 \boxplus \cdots \boxplus \pi_d, \quad \pi_i \in \mathcal{A}^{n_i}(F). \]
The correspondence is such that $\sigma_i \longleftrightarrow \pi_i$, for each $i = 1, \ldots, d$; and, we have equality of Artin and automorphic $L$-functions
\[
L(s, \sigma_i) = L(s, \pi_i).
\] 
This latter form of the Langlands correspondence between semisimple representations on the Galois side and automorphic representations of the general linear group is the one we consider; we again write $\sigma \longleftrightarrow \pi$ in this setting for corresponding representations.

\subsection{}
We are fortunate in the case of function fields that we can extend Theorem \ref{thm:symm:Galois} to irreducible $\ell$-adic representations of $\mathcal{W}_F$ and pass to the automorphic side of the global Langlands correspondence to obtain Theorem \ref{thm:main:cuspidal} below. The general classification result on the automorphic side can be succintly stated if we let $M$ be the maximal cuspidal symmetric power of $\pi$. When $M$ exists as a finite positive integer, it is determined by ${\rm Sym}^{M}(\pi)$ being cuspidal, while ${\rm Sym}^{M+1}(\pi)$ is non-cuspidal. We write $M = \infty$ in case every symmetric power of $\pi$ is cuspidal.

\begin{thm}\label{thm:main:cuspidal}
Let $\pi$ be a cuspidal automorphic representation of ${\rm GL}_2(\mathbb{A}_F)$. Then ${\rm Sym}^6(\pi)$ is cuspidal if and only if $M=\infty$. If ${\rm Sym}^6(\pi)$ is non-cuspidal, then $\pi$ admits the following classification. 
\begin{itemize}
    \item[] M=1: $\pi$ is dihedral.
    \item[] M=2: $\pi$ is tetrahedral.
    \item[] M=3: $\pi$ is octahedral.
    \item[] M=5: $\pi$ is icosahedral.
\end{itemize}
Additionally, ${\rm Sym}^4(\pi)$ is cuspidal if and only ${\rm Sym}^5(\pi)$ is cuspidal.
\end{thm} 
\begin{proof}
Let $\sigma\colon \mathcal{W}_F\to {\rm GL}(V)$ be the irreducible 2-dimensional $\ell$-adic representation attached to $\pi$. Let $\chi$ be a character such that $\sigma\otimes\chi$ can be extended to $G_F$. 

We continue to denote the extension of $\sigma\otimes\chi$ to $G_F$ by $\sigma\otimes\chi$. For every integer $n\geq2$, we consider the representations 
\begin{equation*}
	\sym^n(\sigma\otimes\chi)=\sym^n(\sigma)\otimes\chi^n\colon G_F\to{\rm GL}(\sym^n(V)),
\end{equation*}
and 
\begin{equation*}
	\sym^n(\sigma)\colon \mathcal{W}_F\to{\rm GL}(\sym^n(V)).
\end{equation*}
Observe that $\sym^n(\sigma\otimes\chi)$ is reducible if and only if $\sym^n(\sigma)$ is reducible.

If $\sym^n(\sigma)$ is reducible for some $n$ then  ${\rm proj}((\sigma\otimes\chi)(G_F))$ is finite by Theorem \ref{thm:symm:Galois}. Since $\mathcal{W}_F$ is dense in $G_F$ and the involved representations to ${\rm PGL}_2(\overline{\mathbb{Q}}_\ell)$ are continuous, ${\rm proj}((\sigma\otimes\chi)(\mathcal{W}_F)) = {\rm proj}(\sigma(\mathcal{W}_F))$ is forced to equal ${\rm proj}((\sigma\otimes\chi)(G_F))$.

With these observations, the theorem follows from Theorem \ref{thm:symm:Galois}, via the global Langlands correspondence.
\end{proof}

Over number fields, the cuspidality criterion for symmetric powers is conjectured in \cite{KiSh2002cusp}, see \S\,3 Conjecture therein; one can there find proofs of the characteristic zero statements corresponding to the cases $M=2, 3$ of Theorem~\ref{thm:main:cuspidal}.

Let us next proceed to the automorphic statements corresponding to the detailed criteria obtained along the course of proving our results on the Galois side. In the refined statements, we make use of quadratic base change and automorphic induction, available to us in greater generality thanks to the work of Henniart-Lemaire \cite{HeLe2011}.

\subsection{On cuspidality criteria for the symmetric square, cube and fourth}\label{ss:auto:sym234}
We work with a cuspidal representation $\pi$ of ${\rm GL}_2(\mathbb{A}_F)$, and make use of the following notation:
\[ A^n(\pi) = \sym^n(\pi) \otimes \omega_\pi^{-1}, \]
where $\omega_\pi$ denotes the central character of $\pi$. The case of $n=2$ is the adjoint lift of Gelbart-Jacquet \cite{GeJa1978}, which we denote by ${\rm Ad}(\pi)$.

To begin with, in the case of symmetric square, Galois Theorem~\ref{thm:sym2} allows us to infer that the following statements are equivalent:
\begin{enumerate}
    \item[(i)] $\sym^2(\pi)$ is non-cuspidal. 
    \item[(ii)] $\pi\cong \pi\otimes\chi$ for some non-trivial character $\chi$. 
\end{enumerate}
Note that $\pi$ is dihedral precisely when these conditions are met, by Theorem~\ref{thm:main:cuspidal}. 

From Galois Theorem~\ref{thm:A3}, translated to the automorphic side: if $\pi$ is non-dihedral, then $\sym^3(\pi)$ is non-cuspidal if and only if there exists some non-trivial character $\mu$ such that
\[ \sym^2(\pi)\cong \sym^2(\pi) \otimes \mu; \]
the condition being equivalent to $\pi$ being tetrahedral. When this is the case we have the decomposition
\begin{equation*}
    A^3(\pi)\cong \pi\otimes\mu \boxplus \pi\otimes\mu^2.
\end{equation*}

Now, assume that $\sym^3(\pi)$ is irreducible. Via Galois Theorem~\ref{thm:sym4}, we obtain that $\sym^4(\pi)$ is cuspidal if and only if there exists a non-trivial quadratic character $\chi$ such that
\[ \sym^3(\pi)\cong \sym^3(\pi)\otimes \chi; \]
the condition being equivalent to $\pi$ being octahedral. When this is the case, we get 
\begin{equation*}
    A^4(\pi)\cong \Pi_E^F(\omega_{\pi_E} \mu) \boxplus \sym^2(\pi)\otimes\chi,
\end{equation*}
where $E/F$ is the quadratic extension corresponding to $\chi$ via class field theory and $\mu$ is some cubic character of $\mathbb{A}_E^\times$. Let us explain the notation, where we are assuming $\pi \longleftrightarrow \sigma$, i.e., $\pi\in\mathcal{A}^2(F)$ corresponds to $\sigma\in\mathcal{G}^2_\ell(F)$ under global Langlands; with corresponding base change $\pi_E \longleftrightarrow \sigma_E$. Then we have induced representations on the Galois side and monomial representations on the automorphic side, in particular,
\[ \rm{Ind}_E^F(\omega_{\sigma_E} \mu) \longleftrightarrow \Pi_E^F(\omega_{\pi_E} \mu), \]
where $\mu$ is some cubic character of $G_E$ and $\Pi_E^F$ denotes automorphic induction, cf. \S~IV.5 of \cite{HeLe2011}.

\subsection{On cuspidality criteria involving symmetric sixth}\label{ss:auto:sym6:dec} If $\sym^4(\pi)$ is cuspidal, then Galois Theorem~\ref{thm:sym6:dec} leads us to conclude that the following statements are equivalent:
\begin{enumerate}
    \item [(i)] $\sym^6(\pi)$ is non-cuspidal.
    \item [(ii)] There exists a cuspidal two-dimensional representation $\pi'$ of ${\rm GL}_2(\mathbb{A}_F)$ such that $\sym^5(\pi)\cong {\rm{Ad}}(\pi)\boxtimes\pi'$.
\end{enumerate}
The representation $\pi'$ is uniquely determined up to isomorphism by (ii).

Next, we look into the results corresponding to Galois Proposition~\ref{prop:sym6:chi}. Assume that $\sym^4(\pi)$ is cuspidal and $\sym^6(\pi)$ is non-cuspidal. Then there exists a unique character $\chi$ such that
\begin{equation*}\label{eq:auto:prop:sym6}
\sym^5(\pi)\cong {\rm Ad}(\pi') \boxtimes \pi \otimes \chi,	
\end{equation*}
and we have the decomposition 
\begin{equation*}\label{eq:auto:prop:sym6:dec}
    \sym^6(\pi) \cong \pi \boxtimes \pi' \boxplus {\rm{Ad}}(\pi') \otimes \chi \omega_\pi.
\end{equation*}
Notice that the uniqueness of $\chi$ is obtained via Galois Corollary~\ref{cor:unique:chi}. Furthermore, from Galois Proposition~\ref{prop:gal:sym6:sym2}, we obtain that
\begin{equation*}
    \sym^3(\pi') \cong \sym^3(\pi) \otimes \mu,
\end{equation*}
for $\mu = \omega_{\pi'} \omega_\pi^{-1} \chi$; and, one can also write $\mu=(\eta\,\omega_\pi^2)^3$ for a quadratic character $\eta$.

\bibliographystyle{amsalpha}

\medskip

\noindent{\sc \Small Luis Alberto Lomel\'i, Instituto de Matem\'aticas, Pontificia Universidad Cat\'olica de Valpara\'iso, Blanco Viel 596, Cerro Bar\'on, Valpara\'iso, Chile}\\
\emph{\Small E-mail address: }\texttt{\Small Luis.Lomeli@pucv.cl}

\medskip

\noindent{\sc \Small Francisco Javier Navarro, Instituto de Matem\'aticas, Pontificia Universidad Cat\'olica de Valpara\'iso, Blanco Viel 596, Cerro Bar\'on, Valpara\'iso, Chile}\\
\emph{\Small E-mail address: }\texttt{\Small francisco.navarro.n@pucv.cl}

\end{document}